\documentclass[a4paper,12pt]{amsart}
\usepackage{amssymb}
\usepackage{xypic}
\usepackage[all,cmtip]{xy}
\usepackage{amsmath}
\usepackage{chemarrow}
\usepackage[colorlinks, 
pdfborder=001, 
linkcolor=cyan,
anchorcolor=magenta
citecolor=green
]{hyperref}
\usepackage{mathrsfs}
\usepackage{wasysym}
\usepackage{titletoc}
\usepackage{bm}
\usepackage{bbm}
\usepackage{cite}
\usepackage{subfiles}
\usepackage{extarrows} 
\usepackage{geometry}
\usepackage{colordvi}
\usepackage{color}
\usepackage{stmaryrd}

\allowdisplaybreaks

\DeclareMathOperator{\A}{\mathcal{A}}

\DeclareMathOperator{\Aut}{Aut}

\DeclareMathOperator{\D}{\mathcal{D}}

\DeclareMathOperator{\Ext}{Ext}

\DeclareMathOperator{\gr}{gr}

\DeclareMathOperator{\Hdet}{Hdet}

\DeclareMathOperator{\Hom}{Hom}

\DeclareMathOperator{\id}{id}

\DeclareMathOperator{\Ker}{Ker}
\DeclareMathOperator{\kk}{\Bbbk}

\DeclareMathOperator{\lhu}{\leftharpoonup}
\DeclareMathOperator{\Lt}{{}^{\textbf{L}}\otimes}
\DeclareMathOperator{\M}{\mathcal{M}}

\DeclareMathOperator{\N}{\mathbb{N}}

\DeclareMathOperator{\RHom}{\textbf{R}Hom}
\DeclareMathOperator{\rhu}{\rightharpoonup}

\DeclareMathOperator{\tr}{tr}
\DeclareMathOperator{\vep}{\varepsilon}

\DeclareMathOperator{\vtr}{\vartriangleright}

\DeclareMathOperator{\YD}{\mathcal{YD}}
\DeclareMathOperator{\Z}{\mathbb{Z}}

\newcommand{\To}{\longrightarrow}

\numberwithin{equation}{section}

\theoremstyle{definition}

\newtheorem{thm}{Theorem}[section]
\newtheorem{prop}[thm]{Proposition}
\newtheorem{lem}[thm]{Lemma}

\newtheorem{cor}[thm]{Corollary}
\newtheorem{defn}[thm]{Definition}
\newtheorem{rk}[thm]{Remark}
\newtheorem{ques}[thm]{Question}
\newtheorem{ex}[thm]{Example}

\begin{document}

\title{Skew Calabi--Yau property of faithfully flat Hopf Galois extensions}


\author{Ruipeng Zhu}
\address{School of Mathematics, Shanghai University of Finance and Economics, Shanghai 200433, China}
\email{zhuruipeng@sufe.edu.cn}

\begin{abstract}
	This paper shows that if $H$ is a Hopf algebra and $A \subseteq B$ is a faithfully flat $H$-Galois extension, then $B$ is skew Calabi--Yau provided $A$ and $H$ are.
	Specifically, for cleft extensions $A \subseteq B$, the Nakayama automorphism of $B$ can be derived from those of $A$ and $H$, along with the homological determinant of the $H$-action on $A$.
	This finding is based on the study of the Hopf bimodule structure on $\Ext^i_{A^e}(A, B^e)$.
\end{abstract}




\maketitle



\section*{Introduction}

Calabi--Yau algebras, introduced by Ginzburg \cite{Gin2007}, serve as the noncommutative counterparts of coordinate rings for smooth affine Calabi--Yau varieties.
Their study has been motivated by applications in algebraic geometry, representation theory, and mathematical physics.
Skew Calabi--Yau algebras encompass Artin–Schelter regular algebras \cite{AS1987} and noetherian rigid Gorenstein Hopf algebras of finite global dimension \cite{BZ2008}, sharing homological properties akin to Calabi--Yau algebras.
Each skew Calabi--Yau algebra is associated with a Nakayama automorphism, which is unique up to composition with an inner automorphism.
A skew Calabi–Yau algebra is Calabi--Yau if and only if its Nakayama automorphism is inner.

Numerous studies have explored the Nakayama automorphism and its utilities, see \cite{CWZ2014, LMZ2017} for example.
Informally, \cite{CWZ2014} shows that the Nakayama automorphism governs certain characteristics of the class of Hopf algebras that act on a specified connected graded skew Calabi--Yau algebra.
As established in \cite[Theorem 4.2]{LMZ2017}, the Nakayama automorphism plays a pivotal role as a central element in the Picard group. 
Consequently, gaining a comprehensive understanding and providing an explicit description of the Nakayama automorphism is crucial, as noted in \cite{LMZ2017}.
Simultaneously, the Nakayama automorphism is a subtle invariant that can be challenging to determine.

Identifying a skew Calabi--Yau algebra can also be difficult. To address this, we propose a somewhat more practical criterion to facilitate its determination.

\begin{prop}\label{criterion-hs->sCY} (Proposition \ref{VdB-dual-module-prop})
	Let $A$ be a homologically smooth algebra.
	If
	\begin{enumerate}
		\item[(1)] $\Ext^{i}_{A^e}(A, A^e) = 0$ for all $i \neq d$, and
		\item[(2)] $\Ext^{d}_{A^e}(A, A^e)$ is a finitely generated projective right $A$-module, and
		\item[(3)] $\Ext^{d}_{A^e}(A, A^e)$ is isomorphic to $A$ as a left $A$-module,
	\end{enumerate}
	then $A$ is a skew Calabi--Yau algebra of dimension $d$.
\end{prop}

Calabi--Yau algebras also arise from noncommutative crepant resolutions of Gorenstein algebras \cite{VdB2004}.
For instance, $\mathbb{C}[x_1, \dots, x_n] \# G$, where $G$ is a finite subgroup of $\mathrm{SL}_n(\mathbb{C})$, serves as a non-commutative crepant resolution of $\mathbb{C}[x_1, \dots, x_n]^G$.
Skew group algebras of polynomial rings by finite groups are not uniformly Calabi--Yau. Nonetheless, they consistently exhibit the skew Calabi--Yau property.
The (skew) Calabi--Yau property of the smash product of a skew Calabi--Yau algbera by a Hopf algebra has been examined in \cite{WZ2011}, \cite{LWZ2012}, \cite{YZ2013}, \cite{RRZ2014}, \cite{LMeu2019}.
Let $H$ be a Hopf algebra and $A$ be an $H$-module algebra.
Le Meur established in \cite{LMeu2019} that when both $A$ and $H$ are skew Calabi--Yau, the smash product $A\#H$ is also skew Calabi--Yau, and its Nakayama automorphism can be expressed in terms of those of  $A$ and $H$.
Additionally, Yu and Zhang investigate the skew Calabi--Yau property of the crossed product $A\#_{\sigma}H$, with $\sigma: H \otimes H \to \kk$ a $2$-cocycle, in \cite{YZ2016}.

Smash products and more general crossed products are inherently Hopf Galois extensions.
Let $H$ be a Hopf algebra.
A right $H$-Galois extension is a right $H$-comodule algebra $B$ for which the canonical map $B \otimes_A B \to B \otimes H$ is bijective, where $A = B^{co H}$ is the coinvariant subalgebra.
In noncommutative geometry, faithfully flat Hopf Galois extensions can be viewed as noncommutative principal bundles.
This paper concentrates on examining the skew Calabi--Yau property of such extensions.
Given a faithfully flat Hopf Galois extension $A \subseteq B$ over a Hopf algebra $H$, Stefan’s spectral sequence arises as
\[ \Ext^{p}_{H}(\kk, \Ext^q_{A^e}(A, B^e)) \Longrightarrow \Ext^{p+q}_{B^e}(B, B^e) \]
It is not easy to clarify both the right $H$-module and left $B^e$-module structures on $\Ext^n_{A^e}(A, B^e)$ concurrently.
However, we discover that $\Ext_{A^e}^i(A, B^e)$ admits an $H^e$-comodule structure, thereby endowing it with the structure of a Hopf bimodule in the category $_{B^e}\!\M^{H^e}_{H}$.
This enables us to easily verify the conditions stated in Proposition \ref{criterion-hs->sCY}, which in turn allows us to establish our main theorem.

\begin{thm}\label{main-thm-1} (Theorem \ref{sCY-for-ff-H-ext})
	Let $H$ be a skew Calabi--Yau Hopf algebra of dimension $n$, and $A \subseteq B$ be a faithfully flat $H$-Galois extension.
	If $A$ is skew Calabi--Yau of dimension $d$, then $B$ is skew Calabi--Yau of dimension $(d+n)$.
\end{thm}


The proof of Theorem \ref{main-thm-1} does not yield a concrete characterization of the Nakayama automorphism.
However, by utilizing Schauenburg’s structure theorem for relative Hopf bimodules, it is demonstrated that the category $_{B^e}\!\M^{H^e}_{H}$ is equivalent to a left module category over $\Lambda_1$, where $\Lambda_1$ is a subalgebra of $B^e$.
In the case of a smash product $B= A\#H$, $\Lambda_1$ is isomorphic to $\Delta_1$, as discussed in \cite{Kay2007, LMeu2019}, which has been instrumental in studying the Calabi--Yau property of such products.
For a crossed product $A \#_{\sigma} H$, there exists a natural embedding $H \hookrightarrow \Lambda_1$, allowing the homological determinant to be defined.
Consequently, the Nakayama automorphisms of $A \#_{\sigma} H$  are governed by those of $A$ and $H$, as well as the homological determinant of the $H$-action on $A$.

\begin{thm}\label{main-thm-2} (Theorem \ref{Nak-of-cros-prod})
	Let $H$ be a skew Calabi--Yau algebra of dimension $n$, $B = A \#_{\sigma} H$ be a crossed product of $A$ with $H$.
	If $A$ is skew Calabi--Yau of dimension $d$ with a Nakayama automorphism $\mu_A$, then $B$ is also a skew Calabi--Yau algebra of dimension $n+d$ with a Nakayama automorphism $\mu_B$ which is defined by
	$$\mu_B(a\#h) = \sum \mu_A(a) \Hdet(S^{-2}h_1) \# S^{-2}h_{2} \chi(Sh_{3}),$$
	where $\Hdet$ is the homological determinant with respect to a $\mu_A$-twisted volume as defined in Definition \ref{Hdet-defn} and $\chi: H \to \kk$ is given by $\Ext^n_H(\kk, H) \cong {_{\chi}\!\kk}$.
\end{thm}


The organization of the paper is as follows.

In Section \ref{pre-section}, we review essential notation and recap preliminary findings.
Theorem \ref{criterion-hs->sCY} presents a criterion for a homologically smooth algebra to be skew Calabi--Yau, and its proof is provided in this section.
Section \ref{sCY-HG-ext-section} outlines the basic definitions and properties of faithfully flat Hopf Galois extensions.
We endow $\Ext_{A^e}^i(A, B^e)$ with a Hopf module structure and, leveraging this framework, establish Theorem \ref{main-thm-1} using Stefan’s spectral sequence.
We then consider the Nakayama automorphism for specific Hopf Galois extension cases.
To delve into the Nakayama automorphism for faithfully flat Hopf Galois extensions $A\subseteq B$, we revisit Schauenburg’s structure theorem for relative Hopf bimodules in Section \ref{Hopf-bimod-cats-section}, followed by an examination of the Hopf module structure on $\Ext_{A^e}^i(A, B^e)$.
In Section \ref{Nak-aut-cleft-ext-section}, we describe an $H$-action on $\Ext_{A^e}^i(A, A^e)$ when $B$ is the crossed product $A\#_{\sigma}H$. We compute the homological determinant for the crossed product $A\#_{\sigma}H$. Subsequently, Theorem \ref{main-thm-2} is proven, which expresses the Nakayama automorphism of $A\#_{\sigma}H$ in terms of the Nakayama automorphisms of $A$ and $H$, along with the homological determinant of the $H$-action on $A$.

\section{Preliminaries}\label{pre-section} 

Throughout this paper, $\kk$ is a base field, and all vector spaces and algebras are over $\kk$. Unadorned $\otimes$ means $\otimes_{\kk}$ and $\Hom$ means $\Hom_{\kk}$.
Let $A$ be an algebra over the base field $\kk$.
We denote the opposite algebra of $A$ by $A^{op}$, and denote the enveloping algebra $A \otimes A^{op}$ of $A$ by $A^e$.
Note that the switching operation $a \otimes b \mapsto b \otimes a$ extends to an anti-automorphism of the algebra $A^e$.
We always assume that $\kk$ acts centrally on all bimodules, and then we can identify the right (or left) $A^e$-modules with the $A$-$A$-bimodules.
Usually we work with right modules, and a left $A$-module is viewed as a right $A^{op}$-module.


Given an abelian category $\A$, let $\D(\A)$ denote the derived category of $\A$, and for $* = -, +, b$ let $\D^{*}(\A)$ be the appropriate full subcategory (conventions as in \cite{Wei1994}).
Given an algebra $A$, let $\D(A)$ denote the derived category of (right) $A$-modules. 
A complex $X \in \D^{-}(A)$ is called {\it pseudo-coherent} if $L$ has a bounded-above free resolution $P^{\bullet} = (\cdots \to P^i \to P^{i+1} \to \cdots)$ such that each component $P^i$ is a finitely generated free $A$-module \cite{YZ2006}.
An $A$-module $M$ is called {\it pseudo-coherent} if it is pseudo-coherent as a complex.
A complex is called {\it perfect}, if it is quasi-isomorphic to a bounded complex of finitely generated projective $A$-modules.
Let $\D_{perf}(A^e)$ denote the full subcategory of $\D^{b}(A^e)$ consisting of perfect complexes.

For complexes $M, N \in \D(A)$ with either $M \in \D^{-}(A)$ or $N \in \D^{+}(A)$, let $\RHom_{A}(M, N) \in \D(\kk)$ denote the right derived functor of the Hom functor.
For complexes $M \in \D(A)$ and $N \in \D(A^{op})$ with $M \in \D^{-}(A)$ or $N \in \D^{-}(A^{op})$, let $M \Lt_A N \in \D(\kk)$ denote the left derived functor of the tensor functor.
For the notations of derived categories and derived functors we refer the reader to \cite{Wei1994,Yek2020}.

\subsection{Skew Calabi--Yau algebras and Nakayama automorphism}

An algebra is called {\it homologically smooth}, if $A$ is perfect as a complex of $A^e$-modules.
For any $A^e$-module $M$, the left $A^e$-module structure on $A^e$ induces a $A$-$A$-bimodule structure on $\Hom_{A^e} (M, A^e)$.
For any $f \in \Hom_{A^e}(M, A^e)$, $a, a' \in A$ and $m \in M$,
$$(a f a')(x) = (a \otimes a') f(m) \in A^e.$$
Recall that a $\kk$-algebra $A$ has {\it Van den Bergh duality} (or {\it VdB duality}, for short), if it is homologically smooth, and there exists some invertible $A$-$A$-bimodule $\omega$ and $d \in \mathbb{N}$, such that
$$\RHom_{A^e}(A, A^e) \cong \omega[-d]$$
in $\D^b(A^e)$, and $\omega$ is called the {\it VdB dual module} of $A$.

Let $M$ be an $A$-$A$-bimodule.
If $\mu$ is an algebra endomorphism of $A$, then we let $M_{\mu}$ denote the bimodule $M$ with normal left action, but right action twisted by $\mu$.
That is, $a \cdot m \cdot a' = am\mu(a')$ for all $a, a' \in A$ and $m \in M$.

\begin{defn}\cite[Definition 0.1]{RRZ2014}
	A homologically smooth algebra $A$ is called {\it skew Calabi--Yau} of dimension $d$ for some integer $d \geq 0$, if there exists an algebra automorphism $\mu \in \Aut(A)$ such that
	$$\RHom_{A^e}(A, A^e) \cong A_{\mu}[-d]$$
	in $\D^b(A^e)$.
	Then $\mu$ is called a {\it Nakayama automorphism} of $A$.
	If $\mu$ is an inner automorphism, that is, $\Ext^d_{A^e}(A, A^e) \cong A$ as $A^e$-module, then $A$ is called a Calabi--Yau algebra of dimension $d$.
\end{defn}

The coordinate ring of a smooth affine Calabi--Yau variety is the original and most prototypical example of a Calabi--Yau algebra.
Skew Calabi--Yau algebras are commonly exemplified by Artin--Schelter regular and Hopf algebras, as discussed in the following subsection.

\subsection{Artin--Schelter regular Hopf algebras}

Throughout this paper, $H$ will stand for a Hopf algebra $(H, \Delta, \varepsilon)$ with a bijective antipode $S$.
We use the Sweedler notation $\Delta(h) = \sum h_1 \otimes h_2$ for all $h \in H$.
We recommend \cite{Mon1993} as a basic reference for the theory of Hopf algebras.
We will always assume that the antipode $S$ is bijective. This assumption on $S$ is automatic if $H$ is noetherian Artin--Schelter Gorenstein \cite[Theorem 0.2]{LOWY2018} or skew Calabi--Yau \cite[Proposition 3.4.2]{LMeu2019}.

\begin{defn}
	A Hopf algebra $H$ is called {\it Artin--Schelter} ({\it AS}) {\it Gorenstein} of dimension $d$, if
	\begin{enumerate}
		\item $H$ has left and right injective dimension $d < + \infty$,
		\item $\Ext^i_{H}(\kk,H) \cong \Ext^i_{H^{op}}(\kk, H) \cong
		\begin{cases}
			0, & i \neq d\\
			\kk, & i = d
		\end{cases}$, where $\kk$ is the trivial $H$-module $H/\Ker\varepsilon$.
	\end{enumerate}
	If in addition, $H$ has finite global dimension $d$, then $H$ is called {\it Artin--Schelter} ({\it AS}) {\it regular}.
\end{defn}

Let $H$ be an AS-Gorenstein algebra of dimension $d$.
The {\it right homological integral} \cite{LWZ2007} of $H$ is defined to be the $1$-dimensional $H$-$H$-bimodule $\Ext^d_{H}(\kk, H)$.
Then there exists an algebra homomorphism $\chi: H \to \kk$ such that $\Ext^d_{H}(\kk, H)$ is isomorphic to ${_{\chi}\!\kk}$ as $H$-$H$-bimodules, that is, for any $x \in \Ext^d_{H}(\kk, H)$, $hx = \chi(h)x$ and $xh = \vep(h)x$.
One may define the left homological integral $\Ext^d_{H^{op}}(\kk, H)$ similarly.

As is well known, the skew Calabi--Yau property of a Hopf algebra is also equivalent to the AS regular property. 

\begin{thm}\cite[Lemma 3.3]{LMeu2019}
	The following statements are equivalent for a Hopf algebra $H$.
	\begin{enumerate}
		\item $H$ is homologically smooth.
		\item $\kk$ is a perfect $H$-complex.
	\end{enumerate}
\end{thm}

\begin{thm}\cite[Theorem 1]{LMeu2019}\label{Hopf-VdB<->sCY-thm}
	The following statements are equivalent for a Hopf algebra $H$ with a bijective antipode.
	\begin{enumerate}
		\item $H$ has VdB duality.
		\item $H$ is AS regular, and the trivial $H$-module $\kk$ is pseudo-coherent.
		\item $H$ is skew Calabi--Yau.
	\end{enumerate}
	In each case, $\mu_H$ serves as a Nakayama automorphism of $H$, given by $\mu_H(h) = \sum S^{-2}h_1 \chi(Sh_2)$ for all $h \in H$, where $\Ext^d_{H}(\kk, H) \cong {_{\chi}\!\kk}$.
\end{thm}

Here, we present several examples of skew Calabi--Yau Hopf algebras.

\begin{ex}
	\begin{enumerate}
		\item If $H$ is a finite-dimensional Hopf algebra, it is skew Calabi--Yau if and only if it is semisimple.
		\item Any noetherian affine PI Hopf algebra with finite global dimension $d$ is inherently a $d$-dimensional skew Calabi--Yau algebra \cite[6.2]{BZ2008}.
		In particular, the coordinate ring $\mathcal{O}(G)$ of an affine algebraic group $G$ over a field $\kk$ of characteristic zero is always a Calabi--Yau algebra.
		\item According to \cite[Proposition 6.6]{Zho2024}, any connected Hopf algebra with finite Gelfand--Kirillov dimension $d$ over a field of characteristic 0 is guaranteed to be a skew Calabi--Yau algebra of dimension $d$.
	\end{enumerate}
\end{ex}

\subsection{A criterion for homologically smooth algebras to be skew Calabi--Yau}

In this subsection, we give a useful criterion for homologically smooth algebras to be skew Calabi--Yau.

An object $X$ in $\D^b(A^e)$ is termed a {\it two-side tilting complex} over $A$, if there exists an object $X^{\vee} \in \D^b(A^e)$ satisfying $X \Lt_A X^{\vee} \cong A \cong X^{\vee} \Lt_A X$ in $\D^b(A^e)$.
If $X$ is a two-side tilting complex over $A$, then $X^{\vee} \cong \RHom_A(X, A) \cong \RHom_{A^{op}}(X, A)$ in $\D(A^e)$, and $X$ is a perfect complex in both $\D(A)$ and $\D(A^{op})$.

\begin{lem}\label{VdB-dual-module-lem}
	Let $A$ be a homologically smooth algebra.
	If $\Omega:= \RHom_{A^e}(A, A^e)$ is perfect in both $\D(A)$ and $\D(A^{op})$, then $\Omega$ is a two-side tilting complex over $A$.
\end{lem}
\begin{proof}
	Given that $A$ is homologically smooth, there is a bounded complex $P$ of finitely generated projective $A^e$-modules such that $A \cong P$ in $\D^b(A^e)$.
	Thus, in $\D(A^e)$, $\Omega \cong \Hom_{A^e}(P, A^e)$ and $A \cong P \cong \Hom_{A^e}(\Hom_{A^e}(P, A^e), A^e) \cong \RHom_{A^e}(\Omega, A^e)$.
	This implies
	\begin{align*}
		A & \cong \RHom_{A^e}(\Omega, A^e) & \\
		& \cong \RHom_{A^e}(A, \RHom_{A}(\Omega, A^e)) & \\
		& \cong \RHom_{A^e}(A, A^e \otimes_A \RHom_{A}(\Omega, A)) &  \text{since } \Omega \in \D_{perf}(A) \\
		& \cong \RHom_{A^e}(A, A \otimes \RHom_{A}(\Omega, A)) & \\
		& \cong (A \otimes \RHom_{A}(\Omega, A)) \Lt_{A^e} \RHom_{A^e}(A, A^e) & \text{since } A \in \D_{perf}(A^e) \\
		& \cong \RHom_{A}(\Omega, A) \Lt_A \RHom_{A^e}(A, A^e) & \\
		& \cong \RHom_{A}(\Omega, A) \Lt_A \Omega . &
	\end{align*}
	Similarly, it can be shown that $A \cong \Omega \Lt_A \RHom_{A^{op}}(\Omega, A)$ in $\D(A^e)$.
	Consequently, $\Omega$ a two-side tilting complex over $A$.
\end{proof}

The following result is well known.

\begin{lem}\label{invertible to outer automorphism}
	Let $R$ be a ring, and $\omega$ be an $R$-$R$-bimodule.
	If $\omega$ is a free $R$-module of rank one, then there exists an algebra endmorphism $\mu$ of $R$ for which $\omega$ is isomorphic to $R_{\mu}$ as an $R$-$R$-bimodules.
	Furthermore, if $\omega$ is an invertible $R$-$R$-bimodule, then $\mu$ is an automorphism.
\end{lem}

As a consequence of Lemma \ref{VdB-dual-module-lem} and \ref{invertible to outer automorphism}, we establish the following criterion for a homologically smooth algebra to be skew Calabi--Yau.

\begin{prop}\label{VdB-dual-module-prop}
	Let $A$ be a homologically smooth algebra.
	Suppose that there exists an $A$-$A$-bimodule $\omega$ and an integer $d \in \N$ such that $\RHom_{A^e}(A, A^e) \cong \omega[-d]$ in $\D(A^e)$.
	If $\omega$ is finitely generated and projective as both a left and right $A$-module, then $\omega$ is an invertible $A$-$A$-bimodule.
	Moreover, if $\omega$ is isomorphic to $A$ as either a left or right $A$-module, then $A$ is a skew Calabi--Yau algebra of dimension $d$.
\end{prop}

\section{The skew Calabi--Yau property of flat Hopf Galois extensions}\label{sCY-HG-ext-section}

\subsection{Hopf Galois extensions}

In this subsection we recall the definition and basic properties of an $H$-Galois extension. 
Let $B$ be a right $H$-comodule algebra via an algebra homomorphism $\rho: B \to B \otimes H$.
The {\it coinvariant subalgebra} of the $H$-coaction on $B$ is $\{ b \in B \mid \rho(b) = b \otimes 1 \}$, denoted by $B^{co H}$.
If $(V, \rho_V )$ is a right $H$-comodule, then we use the Sweedler notation $\rho(v) = \sum v_0 \otimes v_1$ for all $v \in V$.

\begin{defn}
	With the notation established above, the extension $A(:= B^{co H}) \subset B$ is referred to as a (right) $H$-{\it Galois extension} if the canonical map
	\[ \beta: B \otimes_A B \to B \otimes H, \quad b' \otimes_A b \mapsto b'\rho(b) = \sum b'b_0 \otimes b_1\]
	is bijective.
	If so, the map $\beta$ is termed the {\it Galois map}.
	A Hopf Galois extension is defined as an $H$-Galois extension where $H$ is a Hopf algebra.
\end{defn}

Let $A$ be an algebra, and $\M_A$ (resp., $_A\!\M$) be the category of right (resp., left) $A$-modules.
Let $B$ be a right $H$-comodule algebra. A relative right-right $(B, H)$-Hopf module is a vector space $M$ with a right $B$-action and a right $H$-coaction $\rho_M$ such that
\[ \rho_M(mb) = \rho_M(m)\rho(b) = \sum m_0b_0 \otimes m_1b_1\]
for all $m \in M$ and $b \in B$.
A relative left-right $(B, H)$-Hopf module is defined similarly.
The category of relative right-right (resp., left-right) $(B, H)$-Hopf modules is denoted by $\M^H_B$ (resp., $_B\!\M^H$).

We require some known facts about Hopf Galois extensions.


\begin{thm}\cite[Theorem I]{Sch1990}\cite[Theorem 5.6]{SS2005}\label{faithful-flat-Hopf-Galois-extension-thm}
	Let $H$ be a Hopf algebra with a bijective antipode, $B$ be a right $H$-comodule algebra and $A := B^{co H}$. Then the following are equivalent.
	
	(1) $B$ is injective as right $H$-comodule, and $\beta: B \otimes_A B \to B \otimes H$ is surjective.
	
	(2) $(- \otimes_A B, (-)^{co H})$ is an equivalence between $\M_A$ and $\M^H_B$. 
	
	(3) $(B \otimes_A -, (-)^{co H})$ is an equivalence between $_A\!\M$ and $_B\!\M^H$. 
	
	(4) $B$ is a faithfully flat left $A$-module, and $A \subseteq B$ is an $H$-Galois extension.
	
	(5) $B$ is a faithfully flat right $A$-module, and $A \subseteq B$ is an $H$-Galois extension.
	
	(6) $A \subseteq B$ is a Hopf Galois extension, and the comodule algebra $B$ is right $H$-equivariantly projective as a left $A$-module, i.e., there exists a right $H$-colinear and left $A$-linear splitting of the multiplication $A \otimes B \to B$.
\end{thm}

Then we say a Hopf Galois extension $A \subseteq B$ is faithfully flat if $B$ is faithfully flat as a left or right $A$-module. A faithfully flat $H$-Galois extension is also called a {\it principal homogeneous space} in \cite{Sch1990} or {\it principal $H$-extension} in \cite{BH2004}.

\subsection{Ulbrich-Miyashita actions on $M^A$ and $M_A$}

Suppose now that $A \subset B$ is an $H$-Galois extension.
Following \cite{Brz1996} we shall call the $\kk$-linear map
\[ \kappa: H \to B \otimes_A B, \quad h \mapsto \beta^{-1}(1 \otimes h)\]
the {\it translation map} associated to the $H$-Galois extension $A \subseteq B$.
For any $h \in H$ we shall use the notation
$$\kappa(h) = \sum \kappa^1(h) \otimes_A \kappa^2(h).$$

The following are key properties of $\kappa$ as outlined in \cite{Sch1990}:
\begin{equation}\label{kappa(hk)}
	\sum\kappa^1(hk) \otimes_A \kappa^2(hk) = \sum \kappa^1(k)\kappa^1(h) \otimes_A \kappa^2(h)\kappa^2(k) \text{ in } B \otimes_A B,
\end{equation}
\begin{equation}\label{akappa(h)=kappa(h)a}
	\sum a\kappa^1(h) \otimes_A \kappa^2(h) = \sum \kappa^1(h) \otimes_A \kappa^2(h)a \text{ in } B \otimes_A B,
\end{equation}
\begin{equation}\label{kappa^1-kappa^2_0-kappa^2_1}
	\sum \kappa^1(h) \otimes_A \kappa^2(h)_0 \otimes \kappa^2(h)_1 = \sum \kappa^1(h_1) \otimes_A \kappa^2(h_1) \otimes h_2 \text{ in } B \otimes_A B \otimes H,
\end{equation}
\begin{equation}\label{kappa^1_0-kappa^2-kappa^1_1}
	\sum \kappa^1(h)_0 \otimes_A \kappa^2(h) \otimes \kappa^1(h)_1 = \sum \kappa^1(h_2) \otimes_A \kappa^2(h_2) \otimes Sh_1 \text{ in } B \otimes_A B \otimes H,
\end{equation}
\begin{equation}\label{b_0-kappa(b_1)}
	\sum b_0 \kappa^1(b_1) \otimes_A \kappa^2(b_1) = 1 \otimes_A b \text{ in } B \otimes_A B,
\end{equation}
\begin{equation}\label{m-kappa=vep}
	\sum \kappa^1(h)\kappa^2(h) = \vep(h)  \text{ in } B,
\end{equation}

Given an $H$-Galois extension $A \subseteq B$, $B \otimes_A B$ is a left $H$-module induced by the module structure on $B \otimes H$.
More precisely, for all $h \in H$ and $x, y \in B$, the action is defined as
\[ h \rhu (x \otimes_{A} y) := \beta^{-1}(h\beta(x\otimes_{A} y)) = \beta^{-1}(h\sum xy_0 \otimes y_1) = \beta^{-1}(\sum xy_0 \otimes hy_1).\]
Using equations \eqref{kappa^1-kappa^2_0-kappa^2_1} and \eqref{m-kappa=vep}, we find
\[\beta\big( \sum x\kappa^1(h) \otimes_A \kappa^2(h)y \big) = \sum x\kappa^1(h)\kappa^2(h)_0y_0 \otimes_A \kappa^2(h)_1y_1 = \sum xy_0 \otimes hy_1, \]
The bijectiveness of $\beta$ then yields
\[ h \rhu (x \otimes_{A} y) = \beta^{-1}(\sum xy_0 \otimes hy_1) = \sum x\kappa^1(h) \otimes_A \kappa^2(h)y. \]
Thus, the Galois map $\beta: B \otimes_A B \to B \otimes H$ is a morphism of $H$-$B^e$-bimodules, where the right $B^e$ module structure on $B \otimes H$ is given by
\[ (x \otimes h) (b \otimes b') = \sum b'xb_0 \otimes hb_1 \]
for all $h \in H$ and $x, b, b' \in B$.

For any $A$-$A$-bimodule $M$, we set 
\[ M^A := \{ m \in M \mid am = ma, \; \forall \, a \in A \}, \;\; \text{ and } \;\; M_A := M/[A, M],\]
where $[A, M]$ is the addtive subgroup of $M$ generated by the elements $am - ma$ for all $a \in A, m \in M$.
Then one can easy to check that
\[ M^A \cong \Hom_{A^e}(A, M), \quad M_A \cong A \otimes_{A^e} M \cong M \otimes_{A^e} A. \]

For any left $B^e$-module $M$, we have the following isomorphisms:
\begin{eqnarray}\label{isoms-of-M^A}
	\qquad M^A = \Hom_{A^e}(A, M) \cong \Hom_{B^e}(A \otimes_{A^e} B^e, M) \cong \Hom_{B^e}(B \otimes_A B, M).
\end{eqnarray}
Thus $M^A$ is a right $H$-module via the left $H$-module structure on $B \otimes_A B$.
More precisely, for all $h \in H$ and $x \in M^A$,
\begin{eqnarray}\label{H-mod-on-M^A}
	x \lhu h = \sum \kappa^1(h) x \kappa^2(h).
\end{eqnarray}
Similarly, we also have the following isomorphisms:
\begin{eqnarray}
	M_A \cong A \otimes_{A^e} M \cong (A \otimes_{A^e} B^e) \otimes_{B^e} M \cong (B\otimes_A B) \otimes_{B^e} M.
\end{eqnarray}
Thus $M_A$ is a left $H$-module via the left $H$-module structure on $B \otimes_A B$.
More precisely, for all $h \in H$ and $x \in M$,
\begin{eqnarray}\label{H-mod-on-M_A}
	h \rhu (x +  [A, M]) = \sum \kappa^2(h) x \kappa^1(h) + [A, M].
\end{eqnarray}

%

These $H$-module structures, known as Ulbrich-Miyashita actions, were previously introduced in \cite{DT1989} and \cite{Ste1995}.
It is worth mentioning that the Ulbrich-Miyashita actions are natural in the sense that, for any $B^e$-modules $M$ and $N$, and any map $f \in \Hom_{B^e}(M, N)$, the induced maps
$$f^0: M^A \To N^A, \;\; \text{ and } \;\; f_0: M_A \To N_A$$
are both $H$-linear.
Therefore, the assignment $M \mapsto M^A$ constitutes a functor from ${_{B^e}\!\M}$ (which is equivalent to ${_B\!\M_B}$) to $\M_H$.

\subsection{Stefan's spectral sequence for flat Hopf Galois extensions}

A Hopf Galois extension $A \subseteq B$ is termed flat if $B$ is flat as both a left and right $A$-module.

For any right $H$-module $V$, we have 
\begin{eqnarray}\label{isoms-of-B-otimes-V}
	B \otimes V \cong V \otimes_H (B \otimes H) \cong V \otimes_H (B \otimes_A B).
\end{eqnarray}
As $B \otimes_A B$ is an $H$-$B^e$-bimodule, $B \otimes V$ becomes a $B^e$-module with the action defined by
\[ (x \otimes v)(b \otimes b') = \sum b'xb_0 \otimes vb_1, \]
for any $b, b', x \in B$ and $v \in V$.
The assignment $V \mapsto B \otimes V$ gives a functor from $\M_H$ to ${_{B^e}\!\M}$, which is left adjoint to $(-)^A$.
Moreover, we have the natural isomorphisms
\begin{equation}\label{adjoint-iso-eq}
	\begin{split}
		\Hom_{H}(V, M^A) \stackrel{\eqref{isoms-of-M^A}}{\cong} & \Hom_{H}(V, \Hom_{B^e}(B \otimes_A B, M)) \\
		\cong \;\; & \Hom_{B^e}(V \otimes_H (B \otimes_A B), M) \\
		\stackrel{\eqref{isoms-of-B-otimes-V}}{\cong} & \Hom_{B^e}(B \otimes V, M),
	\end{split}
\end{equation}
for any left $B^e$-module $M$ and right $H$-module $V$.

For any injective $B^e$-module $I$, the adjoint isomorphism \eqref{adjoint-iso-eq} implies that $I^A$ is an injective $H$-module as the functor $B \otimes - : \M_H \to {_{B^e}\!\M}$ is exact.
Furthermore, since
\[ \Hom_{B^e}(B, -) \cong \Hom_{H}(\kk, \Hom_{A^e}(A, -)) = \Hom_{H}(\kk, -) \circ \Hom_{A^e}(A, -), \]
we derive the following Grothendieck spectral sequence associated with a flat Hopf Galois extension, as established by Stefan.

\begin{thm}\cite[Theorem 3.3]{Ste1995}\label{H-Galois-spectral-sequence}
	Let $A \subseteq B$ be a flat $H$-Galois extension. For any $B^e$-module $M$, we have that
	$$\RHom_{B^e}(B, M) \cong \RHom_{H}(\kk, \RHom_{A^e}(A, M)).$$
\end{thm}

To see whether a Hopf Galois extension $B$ is skew Calabi--Yau, we need to study the $B$-$B$-bimodule structure on the cohomological groups $\Ext^n_{B^e}(B, B^e)$ for all $n \geq 0$.
Stefan's spectral sequenceallows us to focus on the $B^e$-$H$ bimodule structure on $\Ext^n_{A^e}(A, B^e)$ for all $n \geq 0$, see the following corollary which is immediate consequence of Theorem \ref{H-Galois-spectral-sequence}.
For clarity, we provide a proof here.
For the reader's convenience, we include a proof here.

\begin{cor}\label{H-Galois-spectral-sequence-for-B^e}
	Let $A \subseteq B$ be a flat $H$-Galois extension. For any $B^e$-module $M$, in $\D(B^e)$
	\[ \RHom_{B^e}(B, B^e) \cong \RHom_{H}(\kk, \RHom_{A^e}(A, B^e)). \]
\end{cor}
\begin{proof}
	Take an injective $B^e \otimes (B^e)^{op}$-module resolution $E^{\bullet}$ of $B^e$.
	Since $B$ is flat both as a left and right $A$-modules, each component $E^n$ is an injective left $A^e$-module by the natural adjoint isomorphism
	$$\Hom_{A^e}(M, E^n) = \Hom_{A^e}(M, \Hom_{B^e}(B^e, E^n)) \cong \Hom_{B^e}(M \otimes_{A^e} B^e, E^n)$$
	for any left $A^e$-module $M$.
	Then $\RHom_{A^e}(A, B^e)$ is isomorphic to $\Hom_{A^e}(A, E^{\bullet})$ in $\D((H \otimes (B^e)^{op})$.
	Given that $\Hom_{A^e}(A, E^n)$ is an injective $H$-module, we have
	\[ \RHom_{H}(\kk, \RHom_{A^e}(A, B^e)) \cong \Hom_{H}(\kk, \Hom_{A^e}(A, E^{\bullet})). \]
	This is isomorphic to $\Hom_{B^e}(B, E^{\bullet})$ due to the adjoint isomorphism \eqref{adjoint-iso-eq}, thus completing the proof.
\end{proof}

On one hand, taking an injective $B^e \otimes (B^e)^{op}$-module resolution $E^{\bullet}$ of $B^e$ allows us to understand the $H$-module structure on $\Ext^n_{A^e}(A, B^e) = \mathrm{H}^n(\Hom_{A^e}(A, E^{\bullet}))$, yet it does not expose the $B^e$-module structure.
On the other hand, a projective $A^e$-module resolution $P_{\bullet}$ of $A$ provides the opposite insight.
Thus, it is challenging to concurrently clarify the right $H$-module and left $B^e$-module structures on $\Ext^n_{A^e}(A, B^e)$.
Subsequently, we utilize the $H$-comodule structure of $\Ext^n_{A^e}(A, B^e)$ to investigate its bimodule structures.

\subsection{Hopf module structure on $M^A$}

Let $H^* := \Hom(H, \kk)$ be the dual algebra of $H$.
A left $H^*$-module $(V, \rhu)$ is called rational if $H^* \rhu v$ is finite dimensional for any $v \in V$.
A left $H^*$-module $V$ is rational if and only if its $H^*$-action comes from a right $H$-coaction, meaning $V$ is an $H$-comodule satisfying
\[ f \rhu v = \sum_{(v)} v_0 f(v_1) \]
for any $v \in V$ and $f \in H^*$.

Let $N$ be a vector space and $M$ be a right $H$-comodule.
The space $\Hom(N, M)$ forms a left $H^*$-module via the action
$$(h^* \rhu f) (n) = h^* \rhu (f(n))$$
for all $h^* \in H^*$, $f \in \Hom(N, M)$ and $n \in N$.
If $N \in {\M_{A^e}}$ and $M \in {\M^{H^e}_{A^e}}$, then $\Hom_{A^e}(N, M)$ is an $H^*$-submodule of $\Hom(N, M)$.
Moreover, if $N$ is finitely generated as an $A^e$-module, then $\Hom_{A^e}(N, M)$ is an $H^*$-submodule of a direct sum of copies of $M$.
Consequently, $\Hom_{A^e}(N, M)$ is a rational $H^*$-module, that is, implying it is an $H$-comodule.
This leads to the following direct corollary.

\begin{lem}\label{lem-0}
	If $N$ is a finitely generated $A^e$-module, then for any $M \in {_{B^e}\!\M^{H^e}_{B^e}}$, $\Hom_{A^e}(N, M)$ is inherently a Hopf module in ${\M^{H^e}_{B^e}}$.
\end{lem}

Let $M$ be a module and comodule over $H$. We call $M$ a (right-right) Yetter-Drinfeld module if and only if the action and coaction are compatible in the following sence:
\begin{eqnarray}\label{YD-condition}
	\rho(m \lhu h) = \sum (m \lhu h)_0 \otimes (m \lhu h)_1 = \sum m_0 \lhu h_2 \otimes (Sh_1)m_1h_3,
\end{eqnarray}
for all $m \in M$ and $h \in H$.
We denote $\YD^H_H$ as the category of right-right Yetter-Drinfeld modules, with morphisms that are both linear and colinear.

For any $M \in {_B\!\M^H_B}$, it is evident that $M^A$ is an $H$-subcomodule of $M$.
We recall that $M^A$ carries a right $H$-module as per \eqref{H-mod-on-M^A}.
This action and coaction fulfill the Yetter-Drinfeld compatibility condition \eqref{YD-condition}, which means  $(-)^A$ can be promoted to a functor from ${_B\!\M^H_B}$ to $\YD^H_H$, as shown in \cite[Proposition 3.10]{DT1989}.

As $B^e$ cannot be inherently viewed as a Hopf bimodule in ${_B\!\M^H_B}$, we must instead examine the category ${_{B^e}\!\M^{H^e}}$.
Since the antipode $S$ is bijective as we assumed, $H^{op}$ also constitutes a Hopf algebra with antipode $S^{-1}$.
For an $H$-Galois extension $A \subseteq B$, $B^{op}$ is a right $H^{op}$-comodule algebra, and $B^{e}$ is a right $H^{e}$-comodule algebra.
Consequently, $A^{op} \subseteq B^{op}$ is an $H^{op}$-Galois extension.
Since there is an isomorphism
\[ B \otimes_A B \stackrel{\cong}{\To} B^{op} \otimes_{A^{op}} B^{op}, \qquad x \otimes_A y \mapsto y \otimes_{A^{op}} x, \]
the element $\sum \kappa^2(k) \otimes_{A^{op}} \kappa^1(k)$ in $B^{op} \otimes_{A^{op}} B^{op}$ is well defined for any $k \in H$.
Furthermore, we have
\begin{align*}
	\beta \big( \sum \kappa^2(S^{-1}h) \otimes_{A^{op}} \kappa^1(S^{-1}h) \big) 
	& = \sum \kappa^1(S^{-1}h)_0\kappa^2(S^{-1}h) \otimes \kappa^1(S^{-1}h)_1 \\
	& \xlongequal{\eqref{kappa^1_0-kappa^2-kappa^1_1}} \sum \kappa^1(S^{-1}h_1)\kappa^2(S^{-1}h_1) \otimes h_2 \\
	& \xlongequal{\eqref{m-kappa=vep}} 1 \otimes h.
\end{align*}
Thus, the translation map $\kappa_{B^{op}}$ for the $H^{op}$-Galois extension $A^{op} \subseteq B^{op}$ satisfies
\[ \kappa_{B^{op}}(h) = \sum \kappa_{B^{op}}^1(h) \otimes_{A^{op}} \kappa_{B^{op}}^2(h) = \sum \kappa^2(S^{-1}h) \otimes_{A^{op}} \kappa^1(S^{-1}h) \in B^{op} \otimes_{A^{op}} B^{op}. \]
Then $A^{e} \subseteq B^{e}$ is an $H^{e}$-Galois extension, with the translation map $\kappa_{B^{e}}$ defined as
\begin{equation}\label{translation-map-B^e-eq}
	\begin{split}
		\kappa_{B^{e}}(h \otimes k) = & \sum \kappa_{B^{e}}^1(h \otimes k) \otimes \kappa_{B^{e}}^2(h \otimes k) \\
		= & \sum \kappa^1(h) \otimes \kappa_{B^{op}}^1(k) \otimes_{A^e} \kappa^2(h) \otimes \kappa_{B^{op}}^2(k) \\
		= & \sum \kappa^1(h) \otimes \kappa^2(S^{-1}k) \otimes_{A^e} \kappa^2(h) \otimes \kappa^1(S^{-1}k) \in B^{e} \otimes_{A^{e}} B^{e}.
	\end{split}
\end{equation}

In the subsequent discussion, $H$ is considered a right $H^e$-comodule algebra through the coaction
\begin{equation}\label{H-right-H^e-comod-alg}
	\rho_H : H \To H \otimes H^{e}, \qquad h \mapsto \sum h_2 \otimes (h_3 \otimes Sh_1).
\end{equation}

\begin{lem}\label{B-otimes-H-mod-str}
	The mapping $M \mapsto M^A$ constitutes a functor from ${\M^{H^e}_{B^{e}}}$ to ${\M^{H^e}_H}$.
	Furthermore, $(-)^A: M \mapsto M^A$ extends to a functor form ${_{B^{e}}\!\M^{H^e}_{B^e}}$ to ${_{B^e}\!\M^{H^e}_H}$.
\end{lem}
\begin{proof}
	For any $M \in {\M^{H^e}_{B^{e}}}$, $M$ is naturally a $B$-$B$-bimodule, and $M^A$ is an $H^e$-subcomodule of $M$.
	Recall that $M^A$ is endowed with a right $H$-module structure indicated in \eqref{H-mod-on-M^A}.
	For any $m \in M^A$, observe that $\rho(m) = \sum m_0 \otimes m_1 \in M \otimes H^e$.
	For any $h \in H$, it follows that
	\begin{align*}
		\rho (m \lhu h) & \xlongequal[]{\eqref{H-mod-on-M^A}} \rho (\sum \kappa^1(h) m \kappa^2(h)) \\
		& = \sum \kappa^1(h)_0 m_0 \kappa^2(h)_0 \otimes \kappa^1(h)_1 m_1 \kappa^2(h)_1 \\
		& \xlongequal[]{\eqref{kappa^1-kappa^2_0-kappa^2_1}, \eqref{kappa^1_0-kappa^2-kappa^1_1}} \sum \kappa^1(h_2) m_0 \kappa^2(h_2)  \otimes (Sh_1) m_1 h_3 \\
		& = \sum (m_0 \lhu h_2) \otimes (m_1(h_3 \otimes Sh_1) ) \\
		& = \rho_M(m) \rho_{H^e}(h),
	\end{align*}
	confirming that $M^A$ belongs to ${\M^{H^e}_H}$.
	Moreover, for any $M \in {_{B^e}\!\M^{H^e}_{B^e}}$, $M^A$ is also a left $B^e$-module and the $B^e$-action is compatible with the $H^e$-coaction.
	Therefore, $M^A$ is a Hopf module in ${_{B^e}}\!\M^{H^e}_H$.
\end{proof}

\begin{lem}\label{H-otimes-B-mod-str}
	The assignment $M \mapsto M_A = M/[A, M]$ becomes a functor form ${_{B^{e}}\!\M^{H^e}}$ to ${_H\!\M^{H^e}}$.
	Furthermore, $(-)_A: M \mapsto M_A$ extends to a functor form ${_{B^{e}}\!\M^{H^e}_{B^e}}$ to ${_H\!\M^{H^e}_{B^e}}$.
\end{lem}
\begin{proof}
	For any $M \in {_{B^{e}}\!\M^{H^e}}$, $M_A$ is a quotient $H^e$-comodule of $M$.
	Recall that $M^A$ is a right $H$-module via \eqref{H-mod-on-M^A}.
	For all $m \in M$ and $h \in H$, we have
	\begin{align*}
		\rho (h \rhu (m + [A, M])) & \xlongequal[]{\eqref{H-mod-on-M_A}} \rho (\sum \kappa^2(h) m \kappa^1(h) + [A, M]) \\
		& = \sum (\kappa^1(h)_0 m_0 \kappa^2(h)_0 + [A, M]) \otimes \kappa^1(h)_1 m_1 \kappa^2(h)_1 \\
		& \xlongequal[]{\eqref{kappa^1-kappa^2_0-kappa^2_1}, \eqref{kappa^1_0-kappa^2-kappa^1_1}} \sum (\kappa^1(h_2) m_0 \kappa^2(h_2) + [A, M]) \otimes (Sh_1) m_1 h_3 \\
		& = \sum (h_2 \rhu (m_0 + [A, M]) ) \otimes ((h_3 \otimes Sh_1) m_1) \\
		& = \rho_{H^e}(h) \rho_M(m + [A, M]).
	\end{align*}
	showing that $M_A \in {_H\!\M^{H^e}}$.
	Moreover, for any $M \in {_{B^e}\!\M^{H^e}_{B^e}}$, $M_A$ is also a right $B^e$-module and the $B^e$-action satisfies the corresponding compatibility condition. Thus $M_A$ is a Hopf module in ${_H\!\M^{H^e}_{B^e}}$.
\end{proof}

Further, consider the following commutative diagram of forgetful functors:
\[ \xymatrix{{_{B^e}\!\M^{H^e}_H} \ar[r] & {_{B^{e}}\!\M^{H \otimes \kk}_H} \ar@{=}[r] & {_{B^{e}}\!\M^{H}_H} \ar[rr] \ar[d] && {_{A \otimes B^{op}}\!\M^{H}_H} \ar[d] \\ && {_{B^{e}}\!\M^{H}} \ar[rr] && {_{A \otimes B^{op}}\!\M^{H}}. } \]
Notice that $A \otimes B^{op}$ is a trivial $H(= H \otimes \kk)$-comodule algebra.
Thus, for any $M \in {_{A \otimes B^{op}}\!\M^{H}_H}$, $M^{co H}$ constitutes an $A \otimes B^{op}$-submodule of $M$.
For any $N \in {_{A \otimes B^{op}}\!\M}$, $N \otimes H$ forms a Hopf bimodule in ${_{A \otimes B^{op}}\!\M^{H}_H}$, defined by
\[ \rho_H(n \otimes h) = \sum_{(h)} n \otimes h_1 \otimes h_2, \quad (a \otimes b)(n \otimes h) h' = (a \otimes b) n \otimes hh', \]
where $a \in A$, $b \in B$, $h, h' \in H$ and $n \in N$.

Then we have the following lemma.

\begin{lem}\label{Hopf-bimod-cat-str-lem}
	Let $H$ be a Hopf algebra and $B$ a right faithfully flat $H$-Galois extension of $A := B^{co H}$.
	Then we have the following category equivalences
	\begin{enumerate}
		\item ${_{A \otimes B^{op}}\!\M^{H}_H} \cong {_{A \otimes B^{op}}\!\M}$ via $M \mapsto M^{co H}$ and $N \otimes H \mapsfrom N$,
		\item ${_{B^e}\!\M^{H}} \cong {_{A \otimes B^{op}}\!\M}$ via $X \mapsto {X}^{co H}$ and $B \otimes_A N \mapsfrom N$,
	\end{enumerate}
\end{lem}
\begin{proof}
	(1) The result is straightforward, given the equivalence ${\M^{H}_H} \cong {\M_{\kk}}$ with the mappings $M \mapsto M^{co H}$ and $N \otimes H \mapsfrom N$.
	
	(2) As $A \otimes B^{op} \subseteq B \otimes B^{op}$ is a faithfully flat $H$-Galois extension (with $H = H \otimes \kk$), the assertion is a direct consequence of Theorem \ref{faithful-flat-Hopf-Galois-extension-thm}.
	
\end{proof}

\subsection{Hopf module structure on $\Ext^n_{A^e}(A, B^e)$}

\begin{lem}\label{Hopf-bimod-cat-is-Grothendieck}
	Let $H$ be a Hopf algebra, $R$ and $T$ be two $H$-comodule algebras.
	Then ${_R\!\M^H_T}$ has enough injective objects.
\end{lem}
\begin{proof}
	For any $R$-$T$-bimodule $M$, $M \otimes H$ is a Hopf bimodule in ${_R\!\M^H_T}$, where the $R$-$T$-bimodule structure defined by $r(m \otimes h)t = \sum r_0mt_0 \otimes r_1ht_1$, and the right $H$-comodule is defined by $\rho(m \otimes h) = \sum m \otimes h_1 \otimes h_2$.
	For any $M \in {_R\!\M_T}$ and $X \in {_R\!\M^H_T}$, there is a natural isomorphism
	\begin{equation}\label{adjoint-Hopf-modules}
		\Phi: \Hom_{_R\!\M_T}(X, M) \stackrel{\cong}{\To} \Hom_{_R\!\M^H_T}(X, M \otimes H), \;\; f \mapsto \big( x \mapsto \sum f(x_0) \otimes x_1 \big),
	\end{equation}
	and the inverse map of $\Phi$ is given by $g \mapsto (\id_M \otimes \varepsilon) \circ g$.
	For any $X \in {_R\!\M^H_T}$, there is an injective $R \otimes T^{op}$-module $I$ with an $R\otimes T^{op}$-linear imbedding $\iota: X \hookrightarrow I$.
	Given that $(\id_M \otimes \vep) \circ \Phi (\iota) = \Phi^{-1} \Phi (\iota) = \iota$, the map $\Phi(\iota): X \hookrightarrow I \otimes H$ is also injective.
	As $I \otimes H$ is an injective object in ${_R\!\M^H_T}$ by the natural isomorphism \eqref{adjoint-Hopf-modules}, we conclude that ${_R\!\M^H_T}$ has enough injective objects.
\end{proof}

\begin{rk}
	Recall that an abelian category $\A$ is {\it Grothendieck}, if it has exact direct limits 
	and a generator, and that any Grothendieck category has enough injective objects.
	Since $\M^H$ is a subcategory of $_{H^*}\!\M$ which is closed under subobjects, quotient objects and coproducts, $\M^H$ is a Grothendieck category (see \cite{DNR2001} for example).
	By analogy, ${_R\!\M^H_T}$ is also classified as a Grothendieck category.
\end{rk}

\begin{prop}\label{RHom-prop}
	Let $H$ be a Hopf algebra with a bijective antipode, $B$ be a flat $H$-Galois extension of $A$.
	\begin{enumerate}
		\item Then the right derived functor $F$ of $(-)^A = \Hom_{A^e}(A, -): {_{B^{e}}\!\M^{H^e}_{B^e}} \to {_{B^e}\!\M^{H^e}_H}$ exists, such that the following diagram 
		\[ \xymatrix{\D^{+}({_{B^{e}}\!\M^{H^e}_{B^e}}) \ar[rr]^{F} \ar[d] && \D({_{B^e}\!\M^{H^e}_H}) \ar[d] \\ \D^{+}(\M_{A^e}) \ar[rr]^{\RHom_{A^e}(A, -)} && \D(\M_{\kk}) } \]
		commutes up to natural isomorphism.
		So the derived functor from $\D^{+}({_{B^{e}}\!\M^{H^e}_{B^e}})$ to $\D({_{B^e}\!\M^{H^e}_H})$ is also denoted by $\RHom_{A^e}(A, -)$.
		\item For any $X \in {_{B^e}\!\M^{H^e}_{A^e}}$, $\RHom_{A^e}(-, X)$ is isomorphic to $X \Lt_{A^e} \RHom_{A^e}(-, A^e)$ as a functor from $\D_{perf}(A^e)^{op}$ to $\D({_{B^e}\!\M^{H^e}})$.
		\item The derived functor $\RHom_{A^e}(-, -)$ is a bifunctor 
		\[\RHom_{A^e}(-, -): \D_{perf}(A^e)^{op} \times \D^{+}({_{B^{e}}\!\M^{H^e}_{B^e}}) \To \D({_{B^e}\!\M^{H^e}}), \]
		and $\Ext_{A^e}^n(Y, X) \cong \mathrm{H}^n(\RHom_{A^e}(Y, X))$.
	\end{enumerate}
\end{prop}
\begin{proof}
	(1) Since ${_{B^{e}}\!\M^{H^e}_{B^e}}$ has enough injective objects by Lemma \ref{Hopf-bimod-cat-is-Grothendieck}, the right derived functor
	\[\RHom_{A^e}(A, -): \D^{+}({_{B^{e}}\!\M^{H^e}_{B^e}}) \To \D({_{B^e}\!\M^{H^e}_H})\]
	exists as per Theorem \cite[Theorem 10.5.6]{Wei1994}.
	Given that $B$ is flat as both a right and a left $A$-module, the forgetful functor from ${_{B^{e}}\!\M^{H^e}_{B^e}}$ to $\M_{A^e}$ maintains injective objects.
	Then the diagram's commutativity follows from the uniqueness of the derived functor.
	
	(2) For any $Y \in \D_{perf}(A^e)$, it is quasi-isomorphic to a bounded complex $P$ of finitely generated projective $A^e$-modules.
	The conclusion follows from
	\begin{align*}
		\RHom_{A^e}(Y, X) & \cong \Hom_{A^e}(P, X) \\
		& \cong X \otimes_{A^e} \Hom_{A^e}(P, A^e) \\
		& \cong X \Lt_{A^e} \Hom_{A^e}(P, A^e) \\
		& \cong X \Lt_{A^e} \RHom_{A^e}(Y, A^e).
	\end{align*}
	
	(3) This follows from a similar proof of \cite[Theorem 10.7.4]{Wei1994}.
\end{proof}

\subsection{Skew Calabi--Yau property of flat Hopf Galois extensions}

In this subsection, let $H$ be a Hopf algebra with a bijective antipode, $B$ be a faithfully flat $H$-Galois extension of $A$.

\begin{prop}\label{key-prop}
	If $A$ is homologically smooth, then
	\[ \RHom_{A^e}(A, B^e) \cong \RHom_{A^e}(A, A^e) \otimes_A B \otimes H \text{ in } \D({\M_{B \otimes H}}).\]
	Further, if $H$ is also homologically smooth, then
	\[ \RHom_{B^e}(B, B^e) \cong \RHom_{A^e}(A, A^e) \otimes_A B \otimes \RHom_{H}(\kk, H) \text{ in } \D({\M_B}).\]
\end{prop}
\begin{proof}
	Since $B$ is faithfully flat both as a left and right $A$-modules, it follows that $B^e$ is faithfully flat as an $A^e$-module.
	$\RHom_{A^e}(A, B^e)$ is a complex in $\D({_{B^e}\!\M^{H^e}_H})$ by Proposition \ref{RHom-prop} (1).
	Given that $A$ is homologically smooth, Proposition \ref{RHom-prop} (2), yields, in $\D({_{B^e}\!\M^{H^e}})$,
	\begin{equation}\label{RHom_A^e-A-B^e-0}
		\begin{split}
			\RHom_{A^e}(A, B^e) & \cong B^e \Lt_{A^e} \RHom_{A^e}(A, A^e) \\
			& \cong B^e \otimes_{A^e} \RHom_{A^e}(A, A^e) \\
			& \cong B \otimes_A (\RHom_{A^e}(A, A^e) \otimes_A B). \\
		\end{split}
	\end{equation}
	Viewing $\RHom_{A^e}(A, B^e)$ as a complex in $\D({_{B^{e}}\!\M^{H}})$, Lemma \ref{Hopf-bimod-cat-str-lem} (2) gives
	\begin{equation}\label{RHom_A^e-A-B^e^coH}
		\begin{split}
			\RHom_{A^e}(A, B^e)^{co H} & \stackrel{\eqref{RHom_A^e-A-B^e-0}}{\cong} \big(B \otimes_{A} \RHom_{A^e}(A, A^e) \otimes_A B \big)^{co H} \\
			& \cong \RHom_{A^e}(A, A^e) \otimes_A B
		\end{split}
	\end{equation}
	in $\D({_{A \otimes B^{op}}\!\M})$.
	Further, by Lemma \ref{Hopf-bimod-cat-str-lem} (1), in $\D({_{A \otimes B^{op}}\!\M^{H}_{H}})$, we have
	\begin{equation}\label{RHom_A^e-A-B^e}
		\begin{split}
			\RHom_{A^e}(A, B^e) & \cong \RHom_{A^e}(A, B^e)^{co H} \otimes H \\
			& \stackrel{\eqref{RHom_A^e-A-B^e^coH}}{\cong}  \RHom_{A^e}(A, A^e) \otimes_A B \otimes H.
		\end{split}
	\end{equation}
	Considering $\RHom_{A^e}(A, B^e)$ in $\D({_{B^{op}}\!\M_H}) (= \D(\M_{B \otimes H}))$, Corollary \ref{H-Galois-spectral-sequence-for-B^e} yields
	\begin{align*}
		\RHom_{B^e}(B, B^e) \cong \;\; & \RHom_{H}(\kk, \RHom_{A^e}(A, B^e)) \\
		\stackrel{\eqref{RHom_A^e-A-B^e}}{\cong} & \RHom_{H}(\kk, \RHom_{A^e}(A, A^e) \otimes_A B \otimes H) \\
		& (\text{since $H$ is homologically smooth.}) \\
		\cong \;\; & (\RHom_{A^e}(A, A^e) \otimes_{A} B \otimes H) \Lt_H \RHom_{H}(\kk, H) \\
		\cong \;\; & \RHom_{A^e}(A, A^e) \otimes_{A} B \otimes \RHom_{H}(\kk, H) & 
	\end{align*}
	in $\D(B)$.
	This completes the proof.
\end{proof}

Similarly, we obtain the following result.

\begin{prop}\label{key-prop'}
	If $A$ is homologically smooth, then
	\[ \RHom_{A^e}(A, B^e) \cong B \otimes_A \RHom_{A^e}(A, A^e) \otimes H \text{ in } \D({_B\!\M_{H}}).\]
	Further, if $H$ is also homologically smooth, then
	\[ \RHom_{B^e}(B, B^e) \cong B \otimes_A \RHom_{A^e}(A, A^e) \otimes \RHom_{H}(\kk, H) \text{ in } \D({_B\!\M}).\]
\end{prop}

Liu, Wu, and Zhu demonstrated in \cite{LWZ2012} that for any $H$-module algebra $R$, the smash product $R\#H$ is homologically smooth if both $R$ and $H$ are homologically smooth. 
Their proof is directly applicable to Hopf Galois extensions with no modifications. Consequently, we can assert the following analogous result.

\begin{prop}\label{homological-smooth-prop} \cite[Proposition 2.11]{LWZ2012}
	Let $A \subseteq B$ be a flat $H$-Galois extension.
	If $A$ and $H$ are homologically smooth, then so is $B$.
\end{prop}


Now we can prove our main result in this paper.

\begin{thm}\label{sCY-for-ff-H-ext}
	Let $H$ be a $n$-dim skew Calabi--Yau Hopf algebra, and $A \subseteq B$ be a faithfully flat $H$-Galois extension.
	
	(1) If $A$ has $d$-dim VdB duality, then $B$ has $(d+n)$-dim VdB duality.
	
	(2) If $A$ is $d$-dim skew Calabi--Yau, then $B$ is $(d+n)$-dim skew Calabi--Yau.
\end{thm}
\begin{proof}
	Let $\omega$ be a VdB dual module of $A$. Hence $\RHom_{A^e}(A, A^e) \cong \omega [-d]$.
	By Proposition \ref{key-prop}, in $\D({\M_B})$,
	\begin{align*}
		\RHom_{B^e}(B, B^e) & \cong \RHom_{A^e}(A, A^e) \otimes_A B \otimes \RHom_{H}(\kk, H) \\
		& \cong \omega [-d] \otimes_A B \otimes \RHom_{H}(\kk, H) \\
		& \cong \omega \otimes_A B [-d-n].
	\end{align*}
	Similarly, we can prove that in $\D({_B\!\M})$,
	\[\RHom_{B^e}(B, B^e) \cong B \otimes_A \omega [-d-n].\]
	Then the conclusion follows from Proposition \ref{VdB-dual-module-prop}.
\end{proof}

In the aforementioned theorem, we did not answer the following question.

\begin{ques}
	What is the precise form of the Nakayama automorphism of $B$ in Theorem \ref{sCY-for-ff-H-ext} (2) ?
\end{ques}

Note that the Nakayama automorphism of any commutative skew Calabi--Yau algebra is the identity map, meaning such algebras are always Calabi--Yau.
Hence we have the following result.

\begin{cor}
	Let $H$ be a $n$-dim skew Calabi--Yau Hopf algebra with a bijective antipode, and $A \subseteq B$ be a faithfully flat $H$-Galois extension of commutative algebras.
	If $A$ is a Calabi--Yau algebra, then so is $B$.
\end{cor}

The geometric version of this corollary states that any principal fiber bundle over a Calabi--Yau variety with a smooth affine algebraic group is also Calabi--Yau.
However, the noncommutative Hopf Galois extension does not preserve the Calabi--Yau property, even when the coinvariant subalgebra is commutative. See the following example.

\begin{ex}
	Let $\mathfrak{g} = \kk x \oplus \kk y$ be a Lie algebra with $[y,x] = x$, $B$ be the enveloping algebra $\mathcal{U}(\mathfrak{g})$ of $\mathfrak{g}$.
	The quotient Hopf algebra $B/(x)$ is denoted by $H$.
	Then $B$ is a right $H$-Galois extension of coinvariant subalgebra $A = \kk[x]$.
	The Nakayama automorphism $\mu$ of $B$ is given by $\mu(x) = x, \mu(y) = y+1$ (see \cite{Yek2000} or \cite{BZ2008}), which is not a inner automorphism.
	Although both $H$ and $A$ are Calabi--Yau algebras, $B$ is not a Calabi--Yau algebra.
\end{ex}

A flat (but not necessarily faithfully flat) Hopf Galois extension may maintain the skew Calabi--Yau property, as demonstrated by the following example.

\begin{ex}\cite[Example 2.4]{Bel2000}
	Consider the Hopf algebra $H = \kk[x]$, where $x$ a primitive element. Define
	$$B = \mathcal{O}(SL(2)) = \kk[a, b, c, d]/(ad-bc-1),$$
	which is a $3$-dim Calabi--Yau algebra.
	$B$ is an $H$-comodule algebra structure via the coaction $\rho: B \to B \otimes H$ given on generators by
	$$\rho(a) = a \otimes 1 + c \otimes x, \quad \rho(b) = b \otimes 1 + d \otimes x, \quad \rho(c) = c \otimes 1, \quad \rho(d) = d \otimes 1.$$
	Note that $B^{co H} = A = \kk[c, d]$.
	Then $A \subseteq B$ is a flat not faithfully flat $H$-Galois extension, as detailed in \cite[Example 2.4]{Bel2000}.
\end{ex}

We do not know whether Theorem \ref{sCY-for-ff-H-ext} still holds in the flat case.

\begin{ques}
	Let $H$ be a skew Calabi--Yau Hopf algebra with a bijective antipode, and $A \subseteq B$ be a flat $H$-Galois extension.
	If $A$ is skew Calabi--Yau, must $B$ also be skew Calabi--Yau ?
\end{ques}

\subsection{Nakayama automorphisms for certain Hopf Galois extensions}

We now present the Nakayama automorphism for specific instances of Hopf Galois extensions.

\subsubsection{Nakayama automorphisms of $\alpha$-Frobenius extensions}

Kreimer and Takeuchi demonstrated that every Hopf Galois extension over a finite-dimensional Hopf algebra is always a Frobenius extension  \cite[Theorem 1.7]{KT1981}.
Recall the notion of $\alpha$-Frobenius extensions, introduced by Nakayama and Tsuzuku \cite{NT1960, NT1961}, which are also referred to as Frobenius extensions of the second kind.
These extensions represent a generalization of the classical Frobenius extensions.

\begin{defn}
	Consider an algebra extension $A \subseteq B$ and an algebra automorphism $\alpha: A \to A$. The extension $A \subseteq B$ is termed an {\it $\alpha$-Frobenius extension} if it satisfies
	\begin{enumerate}
		\item $B$ is a finitely generated projective right $A$-module, and
		\item ${B}$ is isomorphic to $_{\alpha}\!\Hom_A(B, A)$ as $A$-$B$-bimodules.
	\end{enumerate}
\end{defn}

This is equivalent to saying that $B$ is a projective left $A$-module, and ${B}$ is isomorphic to $\Hom_{A^{op}}(B, A)_{\alpha^{-1}}$ as a $B$-$A$-bimodule.
If $\alpha$ is the identity map on $A$, then $A \subseteq B$ is referred to as a {\it Frobenius extension}.

\begin{lem}\label{Frob-Goren-lem}
	Let $A$ be a subalgebra of an algebra $B$.
	Suppose that both $A$ and $B$ are homologically smooth. 
	\begin{enumerate}
		\item If $B$ is a pseudo-coherent right $A$-module, then in $\D(B \otimes A^{op})$,
		\[ \RHom_{A}(B, A) \Lt_{B} \RHom_{B^e}(B, B^e) \cong \RHom_{A^e}(A, A^e) \Lt_A B. \]
		\item If $B$ is a pseudo-coherent left $A$-module, then in $\D(A \otimes B^{op})$,
		\[ \RHom_{B^e}(B, B^e) \Lt_{B} \RHom_{A^{op}}(B, A) \cong B \Lt_A \RHom_{A^e}(A, A^e). \]
	\end{enumerate}
\end{lem}
\begin{proof}
	(1) In $\D(B\otimes A^{op})$, we have the following isomorphisms
	\begin{align*}
		& \RHom_{A}(B, A) \Lt_{B} \RHom_{B^e}(B, B^e) & \\
		\cong & \RHom_{B^e}(B, \RHom_{A}(B, A) \otimes B) & \text{since } B \text{ is homologically smooth} \\
		\cong & \RHom_{B^e}(B, \RHom_{A} (B, A \otimes B)) & \text{since } {B} \text{ is a pseudo-coherent $A$-module} \\
		\cong & \RHom_{A\otimes B^{op}}(B, A \otimes B) & \\
		\cong & \RHom_{A^e}(A, \RHom_{B^{op}}(B, A \otimes B)) & \\
		\cong & \RHom_{A^e}(A, A \otimes B) & \\
		\cong & \RHom_{A^e}(A, A^e) \Lt_A B & \text{since }A \text{ is homologically smooth.}
	\end{align*}
	(2) The proof for the second part is analogous to the first.
\end{proof}

We will show that any homologically smooth $\alpha$-Frobenius extension $B$ of a skew Calabi--Yau algebra $A$ inherently possesses the skew Calabi--Yau property.
Furthermore, we will establish that the restriction of a Nakayama automorphism of $B$ to $A$ constitutes a Nakayama automorphism of $A$.

\begin{prop}\label{Frob->sCY}
	Let $A \subseteq B$ be an $\alpha$-Frobenius extension.
	Suppose that $B$ is homologically smooth.
	If $A$ is a skew Calabi--Yau algebra of dimension $d$, then $B$ is also a skew Calabi--Yau algebra of dimension $d$.
	Moreover, there exists a Nakayama automorphism $\mu$ of $B$ such that $(\mu \circ \alpha)|_A$ is a Nakayama automorphism of $A$.
\end{prop}
\begin{proof}
	Since $A \subseteq B$ is an $\alpha$-Frobenius extension for some $\alpha \in \Aut(A)$, we have $_{\alpha}\!\Hom_A(B, A) \cong {B}$ as $A$-$B$-bimodules and $\Hom_{A^{op}}(B, A)_{\alpha^{-1}} \cong B$ as $B$-$A$-bimodules.
	By Lemma \ref{Frob-Goren-lem}, we obtain the following isomorphisms of $A$-$B$-bimodules
	\[ \Ext^i_{B^e}(B, B^e) \cong {_{\alpha}\!\Hom_A(B, A)} \otimes_B \Ext^i_{B^e}(B, B^e) \cong {_{\alpha}\!\Ext^i_{A^e}(A, A^e)} \otimes_A B. \]
	 Given that $A$ is a skew Calabi--Yau algebra of dimension $d$ with a Nakayama automorphism $\mu_A$, we find $\Ext^i_{B^e}(B, B^e) \cong \begin{cases}
	 	0, & i \neq d \\
	 	{_{\alpha \mu_A^{-1}} B}, & i = d \\
	 \end{cases}$ as $A$-$B$-bimodules.
	Similarly, $\Ext^i_{B^e}(B, B^e) = \begin{cases}
		0, & i \neq d \\
		{B_{\mu_A\alpha^{-1}} }, & i = d \\
	\end{cases}$ as $B$-$A$-bimodules.
	This implies that $B$ is a skew Calabi--Yau algebra by Proposition \ref{VdB-dual-module-prop}. 
	Let $\mu_B$ be a Nakayama automorphism of $B$.
	The $B$-$A$-bimodule $B_{\mu_B\alpha}$ is isomorphic to $B_{\mu_A}$.
	Thus, there exists an invertible element $u \in B$ for which $u^{-1} \mu_B(\alpha(a))u = \mu_A(a)$ holds for all $a \in A$.
	Defining $\mu$ by $b \mapsto u^{-1} \mu_B(b) u$, we confirm that $\mu$ is also a Nakayama automorphism of $B$, and $(\mu \circ \alpha) |_A$ is a Nakayama automorphism of $A$.
\end{proof}

Applying the previous proposition to a Hopf Galois extension over a finite dimensional Hopf algebra, we obtain the following results.

\begin{cor}
	Let $H$ be a finite dimensional semisimple Hopf algebra, and $A \subseteq B$ be an $H$-Galois extension.
	If $A$ is a skew Calabi--Yau algebra of dimension $d$, then $B$ is also a skew Calabi--Yau algebra of dimension $d$, equipped with a Nakayama automorphism $\mu$ such that $\mu|_A$ is a Nakayama automorphism of $A$.
\end{cor}
\begin{proof}
	Given that the $H$-Galois extension $A \subseteq B$ is a Frobenius extension and $B$ is homologically smooth by Proposition \ref{homological-smooth-prop}, the result is a direct consequence of Proposition \ref{Frob->sCY}.
\end{proof}

Additionally, we establish the following connection between an $\alpha$-Frobenius extension and the possession of a VdB dual module.

\begin{prop}
	Let $A \subseteq B$ be an $\alpha$-Frobenius extension.
	Suppose that $A$ is homologically smooth, and $B$ is a faithfully flat $A$-module on both sides.
	If $B$ has a VdB dual module of dimension $d$, then $A$ also has one.
\end{prop}
\begin{proof}
	Given that $B$ is faithfully flat as a right $A$-module, Lemma \ref{Frob-Goren-lem} yields
	\[ \Ext^i_{B^e}(B, B^e) \cong {}_{\alpha} \! \Hom_{A}(B, A) \otimes_B \Ext^i_{B^e}(B, B^e) \cong {}_{\alpha} \! \Ext^i_{A^e}(A, A^e) \otimes_A B \]
	as $A$-$B$-bimodules.
	This implies $\Ext^i_{A^e}(A, A^e) = 0$ for all $i \neq d$ and $\Ext^d_{A^e}(A, A^e) \otimes_A B \cong {_{\alpha^{-1}} \! \Ext^d_{B^e}(B, B^e)}$.
	Write $\omega = \Ext^d_{A^e}(A, A^e)$.
	Since $A$ is a direct summand of $B$ as a left $A$-module by \cite[Proposition 2.11.29]{Row1988}, $\omega$ is a direct summand of the projective left $A$-module ${_{\alpha^{-1}} \! \Ext^d_{B^e}(B, B^e)}$.
	Thus, $\omega$ is a finitely generated projective left $A$-module.
	Analogously, $\omega$ is a finitely generated projective right $A$-module.
	According to Proposition \ref{VdB-dual-module-prop}, $A$ has a VdB dual module $\omega$ of dimension $d$.
\end{proof}

We now proceed to demonstrate that an algebra extension between skew Calabi--Yau algebras is nearly an $\alpha$-Frobenius extension.

\begin{prop}\label{Frob<-sCY}
	Let $A \subseteq B$ be an algebra extension.
	Suppose that $A$ and $B$ are both skew Calabi--Yau algebras of dimension $d$.
	If $B$ is a pseudo-coherent right $A$-module, then $B$ is a projective right $A$-module, and there exists an algebra morphism $\alpha: A \to B$ such that $_{\alpha}\!\Hom_A(B, A) \cong B$ as $A$-$B$-bimodules.
\end{prop}
\begin{proof}
	Let $\mu_A$ and $\mu_B$ be Nakayama automorphisms of $A$ and $B$, respectively.
	By Lemma \ref{Frob-Goren-lem}, in $\D(B \otimes A^{op})$,
	$$\RHom_{A}(B, A) \Lt_{B} B_{\mu_B}[-d] \cong A_{\mu_A}[-d] \Lt_A B.$$
	This yields ${_{\mu_B^{-1} \mu_A} \! \Hom_{A}(B, A)} \cong B$ as $A$-$B$-bimodules and $\Ext^i_{A}(B, A) = 0$ for any $i > 0$.
	Let $n$ be the projective dimension of the right $A$-module $B$.
	Since $A$ has finite global dimension by our assumption, there exists a right $A$-module $N$, such that $\Ext^n_A(B, N) \neq 0$.
	Let $F$ be a free $A$-module with a surjective $A$-module morphism $F \twoheadrightarrow N$.
	This induces a short exact sequence $\Ext^n_A(B, F) \to \Ext^n_A(B, N) \to 0$, leading to $\Ext^n_A(B, F) \neq 0$.
	As $B$ is a pseudo-coherent right $A$-module,
	we have $\Ext^n_A(B, A) \otimes_A F \cong \Ext^n_A(B, F) \neq 0$.
	Consequently, $n = 0$, indicating that $B$ is a projective right $A$-module.
\end{proof}

\subsubsection{Nakayama automotphisms of Hopf Galois extensions of strongly separable algebras}

In the rest of this section, we maintain the assumption that the coinvariant ring $A:=B^{coH}$ is a strongly separable algebra over $\kk$, that is, there exists an element $e:= \sum\limits_{i=1}^n e'_i \otimes e''_i \in A^e$ satisfying
$$\sum_{i=1}^n ae'_i \otimes e''_i = \sum_{i=1}^n e'_i \otimes e''_ia, \qquad \sum_{i=1}^n e'_ie''_i = 1, \qquad \sum_{i=1}^n e''_i \otimes e'_i = e.$$
Such an $e$ is referred to as a symmetric separability idempotent of $A^e$.

It is evident that $\kk$ is strongly separable over itself.
Any Hopf Galois extension of the base field $\kk$ is commonly termed a Hopf Galois object of $H$.
In \cite{Yu2016}, Yu demonstrated that Hopf Galois objects of a skew Calabi--Yau Hopf algebra retain their skew Calabi--Yau nature and explicitly provided their Nakayama automorphism.

We will use the shorthand notation $x \otimes y$ and $x \otimes_A y$ to represent arbitrary elements in $B \otimes B$ and $B \otimes_A B$, respectively.
Before introducing the Nakayama automorphisms of Hopf Galois extensions of strongly separable algebras, we require the following result. 

\begin{lem}\cite[Theorem 4.1]{Agu2000}\label{strongly-separable-lem}
	The map $\tau: (B \otimes B)^A \to B \otimes_A B, \;\; x \otimes y \mapsto y \otimes_A x$ is bijective.
\end{lem}

Observe that $B \otimes_A B$ is also equipped with a right $H$-module defined by
\begin{equation}\label{B-otimes_A-B-right-H-module}
	\begin{split}
		(x \otimes_A y) \lhu h & := \tau\big( \tau^{-1}(x \otimes_A y) \lhu h \big) = \tau\big( (y \otimes x) \lhu h \big) \\
		& \xlongequal{\eqref{H-mod-on-M^A}} \tau\big( \sum \kappa^1(h) y \otimes x \kappa^2(h) \big) \\
		& = \sum x \kappa^2(h) \otimes_A \kappa^1(h) y.
	\end{split}
\end{equation}
Thus, $\tau$ preserves the $H$-module structure.

In the subsequent discussion, we posit that
\[ \sum \kappa^2(h) \otimes \kappa^1(h) = \tau^{-1} \beta^{-1}(1 \otimes h) \in (B \otimes B)^A. \]
The following are useful properties of $\kappa$:
\begin{equation}\label{kappa(hk)'}
	\sum\kappa^2(hk) \otimes \kappa^1(hk) = \sum \kappa^2(h)\kappa^2(k) \otimes \kappa^1(k)\kappa^1(h),
\end{equation}
\begin{equation}\label{kappa^1-kappa^2_0-kappa^2_1'}
	\sum \kappa^2(h)_0 \otimes \kappa^1(h) \otimes \kappa^2(h)_1 = \sum \kappa^2(h_1) \otimes \kappa^1(h_1) \otimes h_2,
\end{equation}
\begin{equation}\label{kappa^1_0-kappa^2-kappa^1_1'}
	\sum \kappa^2(h) \otimes \kappa^1(h)_0 \otimes \kappa^1(h)_1 = \sum \kappa^2(h_2) \otimes \kappa^1(h_2) \otimes Sh_1.
\end{equation}

\begin{lem}\label{B-B-V-mod-str}
	If $V$ is a left $H$-module, then $B \otimes V$ is a $B$-$B$-bimodule which is defined by
	$$b' (x \otimes v) b = \sum b'x\kappa^2(S^{-2}b_1)\kappa^1(S^{-2}b_1) b_0  \otimes S^{-1}b_2 \rhu v.$$
	There is a natural isomorphism $\Phi : (B \otimes_A B) \otimes_H V \to B \otimes V$ in ${_B\!\M_B}$, given by
	$$\Phi(x \otimes_A y \otimes_H v) = \sum x \kappa^2(S^{-2}y_1) \kappa^1(S^{-2}y_1) y_0  \otimes S^{-1}y_2 \rhu v.$$
\end{lem}
\begin{proof}
	First, we demonstrate that $\Phi$ is well defined.
	Given $\kappa^2(h) \otimes \kappa^1(h) \in (B \otimes B)^A$, it is inferred that $\kappa^2(h)\kappa^1(h)$ is an element of the center of $B$.
	Then we have
	\begin{align*}
		\Phi(x \otimes_A ay \otimes_H v) & = \sum x \kappa^2(S^{-2}y_1) \kappa^1(S^{-2}y_1) ay_0  \otimes S^{-1}y_2 \rhu v \\
		& = \sum xa \kappa^2(S^{-2}y_1) \kappa^1(S^{-2}y_1) y_0  \otimes S^{-1}y_2 \rhu v \\
		& = \Phi(x \otimes_A ay \otimes_H v)
	\end{align*}
	for any $a \in A$.
	For any $h \in H$, it holds that
	\begin{align*}
		& \Phi\big( (x \otimes_A y) \lhu h \otimes_H v \big) \\
		& \xlongequal{\eqref{B-otimes_A-B-right-H-module}} \Phi\big( \sum x \kappa^2(h) \otimes_A \kappa^1(h) y \otimes_H v \big) \\
		& \xlongequal{\eqref{kappa^1_0-kappa^2-kappa^1_1'}} \sum x\kappa^2(h_3) \kappa^2\big( (S^{-1}h_2) (S^{-2}y_1) \big) \kappa^1\big( (S^{-1}h_2) (S^{-2}y_1) \big) \kappa^1(h_3) y_0  \otimes (h_1 S^{-1}y_2) \rhu v \\
		& \xlongequal{\eqref{kappa(hk)'}} \sum x \kappa^2(S^{-2}y_1) \kappa^1(S^{-2}y_1) y_0  \otimes h \rhu ( S^{-1}y_2 \rhu v) \\
		& = \Phi\big(x \otimes_A y \otimes_H (h \rhu v) \big).
	\end{align*}
	
	Next, we show that $\Psi: B \otimes V \to (B \otimes_A B) \otimes_H V, x \otimes v \mapsto x \otimes_A 1 \otimes_H v$ is the inverse of $\Phi$ as $\kk$-linear maps.
	One can easily to check that $\Phi\Psi = \id_{B \otimes V}$.
	Note that
	\begin{align*}
		& \quad (\beta \otimes \id_H) \big( \sum x \kappa^2(S^{-2}y_1) \otimes_A \kappa^1(S^{-2}y_1) y_0 \otimes y_2 \big) \\
		& =  \sum x \kappa^2(S^{-2}y_2) \kappa^1(S^{-2}y_2)_0 y_0 \otimes \kappa^1(S^{-2}y_2)_1y_1 \otimes y_3 \\
		& \xlongequal[]{\eqref{kappa^1_0-kappa^2-kappa^1_1'}} \sum x \kappa^2(S^{-2}y_3) \kappa^1(S^{-2}y_3) y_0 \otimes (S^{-1}y_2) y_1 \otimes y_4 \\
		& = \sum x \kappa^2(S^{-2}y_1) \kappa^1(S^{-2}y_1) y_0 \otimes 1 \otimes y_2 \\
		& = (\beta \otimes \id_H) \big( \sum x \kappa^2(S^{-2}y_1) \kappa^1(S^{-2}y_1) y_0 \otimes 1 \otimes y_2 \big).
	\end{align*}
	Since the canonical map $\beta$ is bijective, it follows that
	\begin{eqnarray}\label{eq-1}
		\begin{split}
			& \sum x \kappa^2(S^{-2}y_1) \kappa^1(S^{-2}y_1) y_0 \otimes_A 1 \otimes y_2 \\
			= &\sum x \kappa^2(S^{-2}y_1) \otimes_A \kappa^1(S^{-2}y_1) y_0 \otimes y_2.
		\end{split}
	\end{eqnarray}
	 On the other hand, we can verify that
	\begin{align*}
		\Psi\Phi(x \otimes y \otimes_H v) & = \sum x \kappa^2(S^{-2}y_1) \kappa^1(S^{-2}y_1) y_0 \otimes_A 1 \otimes_H S^{-1}y_2 \rhu v \\
		& \xlongequal[]{\eqref{eq-1}} \sum x \kappa^2(S^{-2}y_1) \otimes_A \kappa^1(S^{-2}y_1) y_0 \otimes_H S^{-1}y_2 \rhu v \\
		& \xlongequal[]{\eqref{B-otimes_A-B-right-H-module}} \sum x \otimes_A y_0 \otimes_H S^{-2}y_1 \rhu (S^{-1}y_2 \rhu v) \\
		& = x \otimes_A y \otimes_H v.
	\end{align*}
	Hence $\Psi\Phi = \id_{(B \otimes_A B) \otimes_H V}$.
	So $B \otimes V$ is a $B$-$B$-bimodule via
	\begin{align*}
		b' (x \otimes v) b := & \Phi\big( b' \Psi(x \otimes v) b \big) = \Phi\big( b' (x \otimes_A 1 \otimes v) b \big) = \Phi( b'x \otimes b \otimes_H v ) \\
		= & \sum b'x\kappa^2(S^{-2}b_1) \kappa^1(S^{-2}b_1) b_0 \otimes S^{-1}b_2 \rhu v,
	\end{align*}
	and $\Phi$ is an isomorphism of $B$-$B$-bimodules.
\end{proof}

The following lemma follows immediately from Lemma \ref{B-B-V-mod-str}.

\begin{lem}\label{B-otimes-B-faithfully-flat}
	The right $H$-module $B \otimes_A B$ is faithfully flat.
\end{lem}

\begin{thm}\label{main-thm-A-finite-dim}
	Let $H$ be a skew Calabi--Yau algebra of dimension $d$, and $B$ be an $H$-Galois extension of a strongly separable algebra $A$..
	Then $B$ is skew Calabi--Yau of dimension $d$ with a Nakayama automorphism $\mu$, which is defined by
	\[ \mu(b) = \sum \kappa^2(S^{-2}b_1) \kappa^1(S^{-2}b_1) b_0 \chi(Sb_2), \]
	where $\chi: H \to \kk$ is given by $\Ext^d_H(\kk, H) \cong {_{\chi}\!\kk}$. 
\end{thm}
\begin{proof}
	Since $A$ is separable, $A$ is a direct summand of $B$ as an $A$-module.
	Hence, $B$ is a faithfully flat $A$-module.
	By Theorem \ref{sCY-for-ff-H-ext}, $B$ is a skew Calabi--Yau algebra of dimension $d$.
	Consider the following isomorphisms:
	\begin{align*}
		& \quad \RHom_{B^e}(B, B^e) & \\
		& \cong \RHom_{H}(\kk, \RHom_{A^e}(A, B^e)) & \text{ by Theorem \ref{H-Galois-spectral-sequence}} \\
		& \cong \RHom_{A^e}(A, B^e) \Lt_H \RHom_{H}(\kk, H) & \\
		& \cong \Hom_{A^e}(A, B^e) \Lt_H \RHom_{H}(\kk, H) & \text{ since $A$ is (strongly) separable} \\
		& \cong (B \otimes_A B) \Lt_H \RHom_{H}(\kk, H) & \text{ by Lemma \ref{strongly-separable-lem} } \\
		& \cong (B \otimes_A B) \Lt_H \Ext^d_H(\kk, H)[-d] & \text{ since $H$ is skew Calabi--Yau}\\
		& \cong (B \otimes_A B) \otimes_H \Ext^d_H(\kk, H)[-d] & \text{ by Lemma \ref{B-otimes-B-faithfully-flat}} \\
		& \cong B \otimes \Ext^d_H(\kk, H)[-d] & \text{ by Lemma \ref{B-B-V-mod-str}.}
	\end{align*}
	Thus, there there is an algebra endomorphism $\mu$ of $B$ such that $B \otimes \Ext^d_H(\kk, H) \cong B_{\mu}$ as $B$-$B$-bimodules, where $\mu$ is given by
	$$\mu(b) = \sum \kappa^2(S^{-2}b_1) \kappa^1(S^{-2}b_1) b_0 \chi(S^{-1}b_2) = \sum \kappa^2(S^{-2}b_1) \kappa^1(S^{-2}b_1) b_0 \chi(Sb_2)$$
	for all $b \in B$.
	Since $B$ is a skew Calabi--Yau algebra, it follows that $B_{\mu}$ is an invertible $B$-$B$-bimodule.
	This implies that $\mu$ is an automorphism.
	Then $\mu$ is a Nakayama automorphism of $B$.
\end{proof}

\section{Hopf bimodules associated to a faithfully flat Hopf Galois extension}\label{Hopf-bimod-cats-section}

Throughout this section, $H$ is a Hopf algebra with a bijective antipode $S$, and $A \subseteq B$ is a faithfully flat $H$-Galois extension.

\subsection{On structural theorems for Hopf bimodule categories}

%

Let $R$ and $T$ be two right $H$-comodule algebras.
Put $L(R, T, H) := (R \otimes T)^{co H}$.
We will abuse notation by using simple tensors $s \otimes t$ instead of arbitrary elements of $R \otimes T$, and note for later use that
\begin{equation}\label{L-defn}
	\begin{split}
		L(R, T, H) = & \{ r \otimes t \; | \, \sum r_0 \otimes t_0 \otimes r_1t_1 = r \otimes t \otimes 1 \} \\
		= & \{ r \otimes t \; | \, \sum r_0 \otimes t \otimes Sr_1 = \sum r \otimes t_0 \otimes t_1 \},
	\end{split}
\end{equation}
as in \cite[Lemma 3.1]{Sch1990}.

It is straightforward to confirm that $L(R, T, H)$ is a subalgebra of $R \otimes T^{op}$, with the multiplication defined as $(r \otimes t)(r' \otimes t') = rr' \otimes t't$.
In general, $L(R, T, H)$ does not form a subalgebra of $R \otimes T$.
Obviously, $R^{co H} \otimes (T^{co H})^{op}$ is contained within $L(R, T, H)$ as a subalgebra.

Let us revisit Schauenburg's structure theorem for relative Hopf bimodules which is a generalization of Theorem \ref{faithful-flat-Hopf-Galois-extension-thm}.

\begin{thm}\cite[Theorem 3.3]{Sch1998}\label{ff-Hopf-Galois-ext-Hopf-bimod-cat-str-thm'}
	Let $H$ be a Hopf algebra and $B$ a right faithfully flat $H$-Galois extension of $A := B^{co H}$.
	Let $R$ be a right $H$-comodule algebra, and put $L:= L(R, B, H)$.
	Then we have quasi-inverse category equivalences
	$${_R\!\M^H_B} \To {_L\!\M}, \;\; N \mapsto N^{co H}, \qquad  {_L\!\M} \To {_R\!\M^H_B}, \;\; M \mapsto M \otimes_A B,$$
	where
	\begin{enumerate}
		\item the $L$-module structure of $N^{co H}$ is defined by
		\[ l \vtr n = r n b, \qquad \forall \; l = r \otimes b \in L \subseteq R \otimes B^{op};\]
		\item for $M \in {_L\!\M}$ we use $\vtr$ to denote the module structure, $N$ is a right $A$-module via $ma = (1 \otimes a) \vtr m$, and the right $B$-module and $H$-comodule structures on $M \otimes_A B$ are those induced by the right tensorand, and the left $R$-module stucture is given by
		$$r(m \otimes_A b) = \sum \big( r_0 \otimes \kappa^1(r_1) \big) \vtr m \otimes_A \kappa^2(r_1) b.$$
	\end{enumerate}
\end{thm}

\begin{thm}\cite[Theorem 3.3]{Sch1998}\label{ff-Hopf-Galois-ext-Hopf-bimod-cat-str-thm}
	Let $R$ be a right $H$-comodule algebra, and put $L:= L(B, R, H)$.
	Then we have quasi-inverse category equivalences
	$${_B\!\M^H_R} \To {_L\!\M}, \;\; N \mapsto N^{co H}, \qquad {_L\!\M} \To {_B\!\M^H_R}, \;\; M \mapsto B \otimes_A M,$$
	where
	\begin{enumerate}
		\item the $L$-module structure of $N^{co H}$ is defined by
		\[ l \vtr n = b n r, \qquad \forall \; l = b \otimes r \in L \subseteq B \otimes R^{op};\]
		\item for \textcolor{red}{$M \in {_L\!\M}$} we use $\vtr$ to denote the module structure, $M$ is a left $A$-module via $am = (a \otimes 1) \vtr m$, and the left $B$-module and $H$-comodule structures on $B \otimes_A M$ are those induced by the left tensorand, and the right $R$-module stucture is given by
		\begin{equation}\label{H-mod-on-B-otimes-M}
			(b \otimes_A m)r = \sum b \kappa^1(S^{-1}r_1) \otimes_A (\kappa^2(S^{-1}r_1) \otimes r_0) \vtr m.
		\end{equation}
	\end{enumerate}
\end{thm}


The following lemmas will be required subsequently.

\begin{lem}
	Let $H$ be a Hopf algebra, $B$ a right faithfully flat $H$-Galois extension of $A := B^{co H}$, and $R$ a right $H$-comodule algebra.
	Then
	\begin{enumerate}
		\item $L(B, R, H)$ is projective as a left $A$-module, and
		\item $L(R, B, H)$ is projective as a right $A$-module.
	\end{enumerate}
\end{lem}
\begin{proof}
	(1) By Theorem \ref{faithful-flat-Hopf-Galois-extension-thm}, a right $H$-colinear and left $A$-linear map $g: B \to A \otimes B$ exists, satisfying $m\circ g = \id_B$ where $m: A \otimes B \to B, a \otimes b \mapsto ab$.
	Consequently, a right $H$-colinear and left $A$-linear map $g \otimes \id_R: B \otimes R \to A \otimes B \otimes R$ is induced, with the property that $(m \otimes \id_R) \circ (g \otimes \id_R) = \id_{B \otimes R}$.
	This yields $L(B, R, H) = (B \otimes R)^{co H}$ as a direct summand of the free $A$-module $(A \otimes B \otimes R)^{co H} = A \otimes L(B, R, H)$.
	Therefore, $L(B, R, H)$ is a projective left $A$-module.
	
	(2) Applying (1) to the faithfully flat $H^{op}$-Galois extension $A^{op} \subseteq B^{op}$, we deduce that $L(R, B, H)$ which equals $L(B^{op}, R^{op}, H^{op})$, is projective as a left $A^{op}$-module, and consequently, as a right $A$-module.
\end{proof}

\begin{lem}\label{mod-str-on-Hom(M, A)-lem}
	Let $H$ be a Hopf algebra, $B$ a right faithfully flat $H$-Galois extension of $A := B^{co H}$, and $R$ a right $H$-comodule algebra. 
	\begin{enumerate}
		\item If $M$ is a left $L(R, B, H)$-module, then $\Hom_{A}(M, A)$ is a left $L(B, R, H)$-module, defined by
		\begin{equation}\label{mod-str-on-Hom(M, A)}
			\big( (b \otimes r) \vtr f) (m) = \sum b f\big( (r_0 \otimes \kappa^1(r_1)) \vtr m \big) \kappa^2(r_1),
		\end{equation}
		for any $b \otimes r \in L(B, R, H)$, $f \in \Hom_{A}(M, A)$ and $m \in M$.
		\item For any left $L(R, B, H)$-module $M$ and $N \in {_B\!\M_B^H}$, $\Hom_{A}(M, N)$ is a Hopf bimodule in ${_B\!\M_R^H}$, where the right $R$-module is given by
		\begin{equation}\label{R-mod-str-on-Hom(M, N)}
			(fr) (m) = \sum f\big( (r_0 \otimes \kappa^1(r_1)) \vtr m \big) \kappa^2(r_1),
		\end{equation}
		for any $b \otimes r \in L(B, R, H)$, $f \in \Hom_{A}(M, N)$ and $m \in M$.
		\item The canonical map
		\[ \Phi: N \otimes_{A} \Hom_{A}(M, A) \To \Hom_{A}(M, N), \qquad n \otimes_{A} f \mapsto \big( m \mapsto nf(m) \big)\]
		is a $B$-$R$-bimodule morphism.
	\end{enumerate}
\end{lem}
\begin{proof}
	(1)
	For any $r \in R$, we have
	\begin{align*}
		& \sum r_0 \otimes \kappa^1(r_2)_0 \otimes_A  \kappa^2(r_2) \otimes r_1 \kappa^1(r_2)_1 \\
		& \xlongequal{\eqref{kappa^1_0-kappa^2-kappa^1_1}} \sum r_0 \otimes \kappa^1(r_3) \otimes_A \kappa^2(r_3) \otimes r_1 Sr_2 \\
		& = \sum r_0 \otimes \kappa^1(r_1) \otimes_A \kappa^2(r_1) \otimes 1,
	\end{align*}
	which implies $\sum r_0 \otimes \kappa^1(r_1) \otimes_A  \kappa^2(r_1)  \in L(R, B, H) \otimes_A B$.
	
	For any $b \otimes r \in L(B, R, H)$, $f \in \Hom_{A}(M, A)$ and $m \in M$, we assert that
	\[ \big( (b \otimes r) \vtr f) (m) = \sum b f\big( (r_0 \otimes \kappa^1(r_1)) \vtr m \big) \kappa^2(r_1) \in A \subseteq B. \]
	This is shown by the equality
	\begin{align*}
		& \big( (b \otimes r) \vtr f) (m)_0 \otimes \big( (b \otimes r) \vtr f) (m)_1 \\
		& = \sum b_0 f\Big( \big( r_0 \otimes \kappa^1(r_1) \big) \vtr m \Big) \kappa^2(r_1)_0 \otimes b_1 \kappa^2(r_1)_0 \\
		& \xlongequal{\eqref{kappa^1-kappa^2_0-kappa^2_1}} \sum b_0 f\Big( \big( r_0 \otimes \kappa^1(r_1) \big) \vtr m \Big) \kappa^2(r_1) \otimes b_1 r_2 \\
		& = \sum b f\Big( \big( r_0 \otimes \kappa^1(r_1) \big) \vtr m \Big) \kappa^2(r_1) \otimes 1 & \text{ since } b \otimes r \in (B \otimes R)^{co H},
	\end{align*}
	implying $\sum b f\Big( \big( r_0 \otimes \kappa^1(r_1) \big) \vtr m \Big) \kappa^2(r_1) \in A$.
	
	Given $a \in A$, we have
	\begin{align*}
		\big( (b \otimes r) \vtr f \big) (ma) & = b f\Big( \big( r_0 \otimes \kappa^1(r_1) \big) \vtr (ma) \Big) \kappa^2(r_1) \\
		& = b f\Big( \big( r_0 \otimes a\kappa^1(r_1) \big) \vtr m \Big) \kappa^2(r_1) \\
		& \xlongequal{\eqref{akappa(h)=kappa(h)a}} b f\Big( \big( r_0 \otimes \kappa^1(r_1) \big) \vtr m \Big) \kappa^2(r_1) a \\
		& = \big( (b \otimes r) \vtr f \big) (m) a,
	\end{align*}
	demonstrating that $(b \otimes r) \vtr f $ is $A$-linear.
	
	For any $b \otimes r, b' \otimes r' \in L(B, R, H)$, we calculate
	\begin{align*}
		 & \quad \Big( (b \otimes r) \vtr \big( (b' \otimes r') \vtr f \big) \Big) (m) \\
		 & = b \big( (b' \otimes r') f \big) \Big( \big( r_0 \otimes \kappa^1(r_1) \big) \vtr m \Big)  \kappa^2(r_1) \\
		& = b b' f \Big( \big( r'_0 \otimes \kappa^1(r'_1) \big) \vtr \big( r_0 \otimes \kappa^1(r_1) \vtr m \big) \Big) \kappa^2(r'_1) \kappa^2(r_1) \\
		& = b b' f \Big( \big( r'_0r_0 \otimes \kappa^1(r_1)\kappa^1(r'_1) \big) \vtr m \Big) \kappa^2(r'_1) \kappa^2(r_1) \\
		& \xlongequal{\eqref{kappa(hk)}} b b' f \Big( \big( r'_0r_0 \otimes \kappa^1(r'_1r_1) \big) \vtr m \Big) \kappa^2(r'_1 r_1) \\
		& = \big( (b b' \otimes r' r) \vtr f \big) (m).
	\end{align*}
	verifying that the module structure \eqref{mod-str-on-Hom(M, A)} is well defined.
	
	(2)
	Given any $f \in \Hom_{A}(M, N)$, $r \in R$, $m \in M$ and $a \in A$.
	\begin{align*}
		(fr) (ma) & = f\Big( \big( r_0 \otimes \kappa^1(r_1) \big) \vtr (ma) \Big) \kappa^2(r_1) \\
		& = f\Big( \big( r_0 \otimes a\kappa^1(r_1) \big) \vtr m \Big) \kappa^2(r_1) \\
		& \xlongequal{\eqref{akappa(h)=kappa(h)a}} f\Big( \big( r_0 \otimes \kappa^1(r_1) \big) \vtr m \Big) \kappa^2(r_1) a \\
		& = (fr) (m) a,
	\end{align*}
	showing that $fr$ is $A$-linear.
	
	For any $r, r' \in R$, we find
	\begin{align*}
		& \quad \big( (fr) r' \big) (m) \\
		& = (fr) \Big( \big( r'_0 \otimes \kappa^1(r'_1) \big) \vtr m \Big) \kappa^2(r'_1) \\
		& = f \Big( \big( r_0 \otimes \kappa^1(r_1) \big) \vtr \big( r'_0 \otimes \kappa^1(r'_1) \vtr m \big) \Big) \kappa^2(r_1) \kappa^2(r'_1) \\
		& = f \Big( \big( r_0r'_0 \otimes \kappa^1(r'_1)\kappa^1(r_1) \big) \vtr m \Big) \kappa^2(r_1) \kappa^2(r'_1) \\
		& \xlongequal{\eqref{kappa(hk)}} f \Big( \big( r_0r'_0 \otimes \kappa^1(r_1r'_1) \big) \vtr m \Big) \kappa^2(r_1 r'_1) \\
		& = \big( f(rr') \big) (m).
	\end{align*}
	confirming that the module structure \eqref{R-mod-str-on-Hom(M, N)} is well defined.
	
	(3)
	Since $\Hom_{A}(M, A)$ is a left $L(B, R, H)$-module, $B \otimes_{A} \Hom_{A}(M, A)$ is a Hopf bimodule in ${_{B}\!\M_R^{H}}$ by Theorem \ref{ff-Hopf-Galois-ext-Hopf-bimod-cat-str-thm}.
	Consequently, $N \otimes_{A} \Hom_{A}(M, A) \cong N \otimes_{B} B \otimes_{A} \Hom_{A}(M, A) $ is indeed a Hopf bimodule in ${_{B}\!\M_R^{H}}$.
	
	It is evident that $\Phi$ is a left $B$-module morphism.
	To demonstrate that $\Phi$ is also a right $R$-module morphism, we calculate
	\begin{align*}
		& \quad \Phi\big((n \otimes_{A} f) r\big) (m) \\
		& \xlongequal{\eqref{H-mod-on-B-otimes-M}} \Phi\Big( \sum n \kappa^1(S^{-1}r_1) \otimes_A \big( \kappa^2(S^{-1}r_1) \otimes r_0 \big) \vtr f \Big) (m) \\
		& = \sum n \kappa^1(S^{-1}r_1) \Big( \big( \kappa^2(S^{-1}r_1) \otimes r_0 \big) \vtr f \Big) (m) \\
		& \xlongequal{\eqref{mod-str-on-Hom(M, A)}} \sum n \kappa^1(S^{-1}r_2) \kappa^2(S^{-1}r_2) f\Big( \big( r_0 \otimes \kappa^1(r_1) \big) \vtr m \Big) \kappa^2(r_1) \\
		& \xlongequal{\eqref{m-kappa=vep}} \sum n f\Big( \big( r_0 \otimes \kappa^1(r_1) \big) \vtr m \Big) \kappa^2(r_1) \\
		& \xlongequal{\eqref{R-mod-str-on-Hom(M, N)}} \big( \Phi(n \otimes_{A^e} f) r \big) (m).
	\end{align*}
	This shows that $\Phi$ is indeed right $R$-linear.
\end{proof}

\subsection{Some subalgebras $\Lambda_i$ of $B^e$}

To further study categories $_{B^e}\!\M^{H^e}_H$ and $_H\!\M^{H^e}_{B^e}$, we now introduce a class of algebras $\Lambda_i$.

Let $(R, \rho)$ be a right $H$-comodule algebra and let $\sigma: H \to K$ be a Hopf algebra map.
Then $R$ also has a $K$-comodule structure via $(\id_R \otimes \sigma) \circ \rho$ such that $R$ is also a right $K$-comodule algebra, denoted by $R^{\sigma}$.
It is obvious that $S^2$ is a Hopf algebra automorphism of $H$.

\begin{lem}\label{L(R, T, H)-lem}
	Let $R$ and $T$ be two right $H$-comodule algebras. 
	\begin{enumerate}
		\item For any Hopf automorphism $\sigma$ of $H$, $L(R, T, H) = L(R^{\sigma}, T^{\sigma}, H)$.
		\item $L(R, T, H) \cong L(T, R^{S^2}, H)^{op}$.
	\end{enumerate}
\end{lem}
\begin{proof}
	(1)
	From \eqref{L-defn}, we deduce that
	\begin{align*}
		L(R, T, H) = & \{ r \otimes t \; | \, \sum r_0 \otimes t \otimes Sr_1 = \sum r \otimes t_0 \otimes t_1 \} \\
		= & \{ r \otimes t \; | \, \sum r_0 \otimes t \otimes \sigma S r_1 = \sum r \otimes t_0 \otimes \sigma t_1 \} \\
		= & \{ r \otimes t \; | \, \sum r_0 \otimes t \otimes S \sigma r_1 = \sum r \otimes t_0 \otimes \sigma t_1 \} \\
		= & L(R^{\sigma}, T^{\sigma}, H).
	\end{align*}
	
	(2)
	According to \eqref{L-defn},
	\[ L(R, T, H) = \{ r \otimes t \; | \, \sum r_0 \otimes t \otimes Sr_1 = \sum r \otimes t_0 \otimes t_1 \} \]
	and
	\[ L(T, R^{S^2}, H) = \{ t \otimes r \; | \, \sum t_0 \otimes r \otimes St_1 = \sum t \otimes r_0 \otimes S^2r_1 \}. \]
	Since $S$ is bijective, we have $L(T, R^{S^2}, H) = \{ t \otimes r \; | \, \sum t_0 \otimes r \otimes t_1 = \sum t \otimes r_0 \otimes Sr_1 \}$.
	Thus, the map
	\[ L(R, T, H) \To L(T, R^{S^2}, H)^{op}, \qquad r \otimes t \mapsto t \otimes r \]
	constitutes an algebra isomorphism.
\end{proof}

Recall that $H$ is a right $H^e$-comodule algebra via \eqref{H-right-H^e-comod-alg}.
For any $t \otimes r \otimes h \in L(T \otimes R^{op}, H, H^e) \subseteq T \otimes R^{op} \otimes H$, then
\begin{equation}\label{T-R-H-eq}
	\sum t_0 \otimes r_0 \otimes h_2 \otimes t_1h_3 \otimes (Sh_1)r_1 = t \otimes r \otimes h \otimes 1 \otimes 1.
\end{equation}
Hence
\begin{align*}
	t \otimes r \otimes h & = \sum t_0 \otimes r_0 \otimes \vep(h_2) \otimes (St_1) t_2 h_3 \otimes \vep((Sh_1)r_1) \\
	& \xlongequal{\eqref{T-R-H-eq}} \sum t_0 \otimes r \otimes \vep(h) \otimes St_1 \otimes \vep(1) \\
	& = \vep(h) \sum t_0 \otimes r \otimes  St_1,
\end{align*}
and 
\begin{align*}
	t \otimes r \otimes h & = \sum t_0 \otimes r_0 \otimes \vep(h_2) \otimes \vep(t_2 h_3) \otimes (S^{-2}r_1)S^{-1}((Sh_1)r_2) \\
	& \xlongequal{\eqref{T-R-H-eq}} \sum t \otimes r_0 \otimes \vep(h) \otimes \vep(1) \otimes S^{-2}r_1 \\
	& = \vep(h) \sum t \otimes r_0 \otimes  S^{-2}r_1.
\end{align*}
It follows from \eqref{L-defn} that $\vep(h) t \otimes r \in L(T, R^{S^{-2}}, H) = L(T^{S^2}, R, H)$.
This gives an algebra isomorphism $L(T \otimes R^{op}, H, H^e) \to L(T^{S^2}, R, H)$ defined by $t \otimes r \otimes h \mapsto \vep(h) t \otimes r$, and the inverse map is given by $t \otimes r \mapsto \sum t_0 \otimes r \otimes St_1$.

Similarly, there is an algebra isomorphism $L(H, R \otimes T^{op}, H^e) \to L(T, R, H)$ defined by $h \otimes r \otimes t \mapsto \vep(h) t \otimes r$, and the inverse map is given by $t \otimes r \mapsto \sum t_1 \otimes r \otimes t_0$.
For any $h \otimes r \otimes t \in L(H, R \otimes T^{op}, H^e) \subseteq H \otimes R \otimes T^{op}$, then
\[ \sum h_2 \otimes r_0 \otimes t_0 \otimes h_3r_1 \otimes t_1Sh_1 = h \otimes r \otimes t \otimes 1 \otimes 1. \]
Hence
\begin{align*}
	 r \otimes t \otimes h = & \sum \vep(h_2) \otimes r_0 \otimes t_0 \otimes h_3r_2S^{-1}r_1 \otimes \vep(t_1Sh_1) \\
	= & \sum \vep(h) \otimes r_0 \otimes t \otimes S^{-1}r_1 \otimes \vep(1) \\
	= & \vep(h) \sum r_0 \otimes t \otimes  S^{-1}r_1,
\end{align*}
and 
\begin{align*}
	r \otimes t \otimes h = & \sum \vep(h_2) \otimes r_0 \otimes t_0 \otimes \vep(h_3r_1) \otimes S^{-1}(t_2Sh_1)t_1 \\
	= & \sum \vep(h) \otimes r \otimes t_0 \otimes \vep(1) \otimes t_1 \\
	= & \vep(h) \sum r \otimes t_0 \otimes  t_1.
\end{align*}
It follows from \eqref{L-defn} that $\vep(h) t \otimes r \in L(T, R, H)$.

\begin{defn}\label{Lambda_i-defn}
	Define $\Lambda_i := L(B^{S^{2i}}, B, H)$ for  $i \in \Z$ and let $\Lambda := \Lambda_0$.
\end{defn}

For any  $i \in \Z$, by Lemma \ref{L(R, T, H)-lem},
\[ \Lambda_i = L(B^{S^{2i}}, B, H) \cong L(B, B^{S^{2(i+1)}}, H)^{op} \cong L(B^{S^{-2(i+1)}}, B, H)^{op} = ({\Lambda_{-i-1}})^{op}. \]
It is evident that $A^e \subseteq \Lambda_i$ implying $A^e$ is a subalgebra of $\Lambda_i$.
Note that $A$ can be seen as a left $\Lambda_i$-module via $(x \otimes y) \vtr a = xay$ for any $a \in A$ and $x \otimes y \in \Lambda_i$.

According to the explanation preceding Definition \ref{Lambda_i-defn}, we have the following result.

\begin{lem}\label{Lambda-equi-lem}
	\begin{enumerate}
		\item[(1)] There is an algebra isomorphism $\Lambda_i \stackrel{\cong}{\to} L(B^{S^{2(i-1)}} \otimes B^{op}, H, H^e)$ defined by $x \otimes y \mapsto \sum x_0 \otimes y \otimes Sx_1$.
		\item[(2)] There is an algebra isomorphism $\Lambda_i \stackrel{\cong}{\to} L(H, B \otimes (B^{S^{2i}})^{op}, H^e)$ defined by $x \otimes y \mapsto \sum x_1 \otimes y \otimes x_0$.
	\end{enumerate}
\end{lem}

The subsequent theorems are consequences of Theorem \ref{ff-Hopf-Galois-ext-Hopf-bimod-cat-str-thm'} and \ref{ff-Hopf-Galois-ext-Hopf-bimod-cat-str-thm}.

\begin{thm}\label{HG-bimod-str-thm-1}
	Let $H$ be a Hopf algebra and $B$ a right faithfully flat $H$-Galois extension of $A := B^{co H}$.
	Then we have quasi-inverse category equivalences
	\[ {_{B^e}\!\M^{H^e}_H} \To {_{\Lambda_1}\!\M}, \;\; N \mapsto N^{co H^e}, \qquad {_{\Lambda_1}\!\M} \To {_{B^e}\!\M^{H^e}_H}, \;\; M \mapsto B^e \otimes_{A^e} M, \]
	where
	\begin{enumerate}
		\item the $\Lambda_1$-module structure of $N^{co H^e}$ is defined by
		\[ \lambda \vtr n = \sum (x_0 \otimes y) n (Sx_1), \qquad \forall \; \lambda = x \otimes y \in \Lambda_1 \subseteq B^e;\]
		\item for $M \in {_{\Lambda_1}\!\M}$ we use $\vtr$ to denote the module structure, $M$ is a left $A^e$-module via $(a \otimes a')m = (a \otimes a' \otimes 1) \vtr m$, and the left $B^e$-module and $H^e$-comodule structures on $B^e \otimes_{A^e} M$ are those induced by the left tensorand, and the right $H$-module stucture is given by
		\begin{equation}\label{H-mod-on-B^e-otimes-M}
			(x \otimes y \otimes_{A^e} m) h = \sum x \kappa^1(S^{-1}h_2) \otimes \kappa^2(Sh_1)y \otimes_{A^e} \big( \kappa^2(S^{-1}h_2) \otimes \kappa^1(Sh_1) \big) \vtr m.
		\end{equation}
	\end{enumerate}
\end{thm}
\begin{proof}
	Recall that $A^e \subseteq B^e$ is an $H^e$-Galois extension.
	The conclusion is then drawn from Theorem \ref{ff-Hopf-Galois-ext-Hopf-bimod-cat-str-thm} and Lemma \ref{Lambda-equi-lem}.
	For the reader’s benefit, we will examine the $H$-module structure on $B^e \otimes_{A^e} M$.
	The isomorphism $L(B^e, H, H^e) \to \Lambda_1, x \otimes y \otimes h \mapsto \vep(h) x \otimes y$ as stated in Lemma \ref{Lambda-equi-lem}, indicates that $M$ is also endowed with a left $L(B^e, H, H^e)$-module structure.
	We denote this module structure by $\vtr$ as well.
	Thus, our notation leads to the expression $(x \otimes y \otimes h) \vtr m = \vep(h) (x \otimes y) \vtr m$.
	Given $x, y \in B$, $h \in H$ and $m \in M$, we compute
	\begin{align*}
		& (x \otimes y \otimes_{A^e} m) h \\
		& \xlongequal{\eqref{H-mod-on-B-otimes-M}} \sum (x \otimes y) \kappa^1_{B^e}\big( S^{-1}_{H^e}(h_3 \otimes Sh_1) \big) \otimes_{A^e} \Big( \kappa^2_{B^e}\big( S^{-1}_{H^e}(h_3 \otimes Sh_1) \big) \otimes h_2 \Big) \vtr m \\
		& = \sum (x \otimes y) \kappa^1_{B^e}(S^{-1}h_3 \otimes S^2h_1) \otimes_{A^e} \big( \kappa^2_{B^e}(S^{-1}h_3 \otimes S^2h_1) \otimes h_2 \big) \vtr m \\
		& \xlongequal{\eqref{translation-map-B^e-eq}} \sum (x \otimes y) \big( \kappa^1(S^{-1}h_3) \otimes \kappa^2(Sh_1) \big) \otimes_{A^e} \big( \kappa^2(S^{-1}h_3) \otimes \kappa^1(Sh_1) \otimes h_2 \big) \vtr m \\
		& = \sum x \kappa^1(S^{-1}h_2) \otimes \kappa^2(Sh_1)y \otimes_{A^e} \big( \kappa^2(S^{-1}h_2) \otimes \kappa^1(Sh_1) \big) \vtr m.
	\end{align*}
\end{proof}

\begin{thm}\label{HG-bimod-str-thm-2}
	Let $H$ be a Hopf algebra and $B$ a right faithfully flat $H$-Galois extension of $A := B^{co H}$.
	Then we have quasi-inverse category equivalences
	\[ {_H\!\M^{H^e}_{B^e}} \To {_{\Lambda}\!\M}, \;\; N \mapsto N^{co H^e}, \qquad  {_{\Lambda}\!\M} \To {_H\!\M^{H^e}_{B^e}}, \;\; M \mapsto M \otimes_{A^e} B^e, \]
	where
	\begin{enumerate}
		\item the $\Lambda$-module structure of $N^{co H^e}$ is defined by
		\[ \lambda \vtr n = \sum x_1 n (y \otimes x_0), \qquad \forall \; \lambda = x \otimes y \in \Lambda \subseteq B^e;\]
		\item for $M \in {_{\Lambda}\!\M}$ we use $\vtr$ to denote the module structure, $M$ is a right $A^e$-module via $m(a \otimes a') = (a \otimes a) \vtr m$, and the right $B^e$-module and $H^e$-comodule structures on $M \otimes_{A^e} B^e$ are those induced by the right tensorand, and the left $H$-module stucture is given by
		\[ h(m \otimes_{A^e} x \otimes y) = \sum \big(\kappa^2(h_1) \otimes \kappa^1(h_2) \big) \vtr m \otimes_{A^e} \kappa^2(h_2) x \otimes y \kappa^1(h_1). \]
	\end{enumerate}
\end{thm}
\begin{proof}
	The proof parallels that of Theorem \ref{HG-bimod-str-thm-1}.
\end{proof}
The subsequent lemma is a direct consequence of Lemma \ref{mod-str-on-Hom(M, A)-lem}, Lemma \ref{Lambda-equi-lem}, Theorem \ref{HG-bimod-str-thm-1} and Theorem \ref{HG-bimod-str-thm-2}.

\begin{lem}\label{mod-str-on-Hom(M, A^e)-lem}
	Let $H$ be a Hopf algebra and $B$ a right faithfully flat $H$-Galois extension of $A := B^{co H}$. 
	\begin{enumerate}
		\item For any left $\Lambda$-module $M$, $\Hom_{A^e}(M, A^e)$ is a left $\Lambda_{1}$-module, defined by
		\begin{equation}\label{mod-str-on-Hom(M, A^e)}
			\begin{split}
				& \big( (x \otimes y) \vtr f \big) (m) \\
				& = \sum (x_0 \otimes y) f\big( \kappa^2(Sx_2) \otimes (\kappa^1(Sx_1)) \vtr m \big) \big( \kappa^2 (Sx_1) \otimes \kappa^1(Sx_2) \big),
			\end{split}
		\end{equation}
		for any $x \otimes y \in \Lambda$, $f \in \Hom_{A^e}(M, A^e)$ and $m \in M$.
		\item For any $M \in {_{\Lambda}\!\M}$ and $N \in {_{B^e}\!\M_{B^e}^{H^e}}$, $\Hom_{A^e}(M, N)$ is a $B^e$-$H$-bimodule, where the right $H$-module is given by
		\begin{equation}\label{mod-str-on-Hom_{A^e}(M, N)}
			(fh)(m) = \sum f \Big( \big( \kappa^2(h_1) \otimes \kappa^1(h_2) \big) \vtr m \Big) \big( \kappa^2(h_2) \otimes \kappa^1(h_1) \big)
		\end{equation}
		for any $f \in \Hom_{A^e}(M, N)$, $h \in H$ and $m \in M$.
		\item The canonical map
		\[ \Phi: N \otimes_{A^e} \Hom_{A^e}(M, A^e) \To \Hom_{A^e}(M, N), \qquad n \otimes_{A^e} f \mapsto \big( m \mapsto nf(m) \big)\]
		is a $B^e$-$H$-bimodule morphism.
	\end{enumerate}
\end{lem}

For any $N \in  {_{B^e}\!\M_{B^e}^{H^e}}$ and $f \in \Hom_{A^e}(A, N)$, we have
\begin{align*}
	(fh)(1_A) & = \sum f \Big( \big( \kappa^2(h_1) \otimes \kappa^1(h_2) \big) \vtr 1_A \Big) \big( \kappa^2(h_2) \otimes \kappa^1(h_1) \big) \\
	& = \sum \kappa^1(h_1) f \big( \kappa^2(h_1) 1_A \kappa^1(h_2) \big) \kappa^2(h_2) \\
	& \;\; \text{ (since } \sum \kappa^1(h_1) \otimes_A \kappa^2(h_1) \kappa^1(h_2) \otimes_A \kappa^2(h_2) \in B \otimes_A A \otimes_A B) \\
	& = \sum \kappa^1(h_1) f(1_A)  \kappa^2(h_1) \kappa^1(h_2)\kappa^2(h_2) \\
	& \xlongequal{\eqref{m-kappa=vep}} \sum \kappa^1(h) f(1_A)  \kappa^2(h) \\
	& \xlongequal{\eqref{H-mod-on-M^A}} f(1_A) \lhu h.
\end{align*}
Consequently, the right $H$-module structure on $\Hom_{A^e}(A, N) \cong N^A$, as defined in \eqref{mod-str-on-Hom_{A^e}(M, N)}, aligns with that in \eqref{H-mod-on-M^A}.

\begin{prop}\label{B^e-otimes_{A^e}-Ext^i_{A^e}(M, A^e)-prop}
	Let $H$ be a Hopf algebra and $B$ a right faithfully flat $H$-Galois extension of $A := B^{co H}$. 
	\begin{enumerate}
		\item For any $M \in {_{\Lambda}\!\M}$ and $i > 0$, $\Ext^i_{A^e}(M, A^e)$ is a left $\Lambda_1$-module.
		\item If $M$ is perfect as an $A^e$-complex, then $\Ext^i_{A^e}(M, B^e) \cong B^e \otimes_{A^e} \Ext^i_{A^e}(M, A^e)$ as $B^e$-$H$-bimodules, where the right $H$-module on $B^e \otimes_{A^e} \Ext^i_{A^e}(M, A^e)$ is given by \eqref{H-mod-on-B^e-otimes-M}.
	\end{enumerate}
\end{prop}
\begin{proof}
	(1)
	Taking a projective resolution $P_{\bullet}$ of the $\Lambda$-module $M$, by Lamma \ref{mod-str-on-Hom(M, A^e)-lem} (1), $\Hom_{A^e}(P_{\bullet}, A^e)$ is a complex of $\Lambda_1$-modules.
	Since $\Lambda$ is a projective $A^e$-module, $\Ext^i_{A^e}(M, A^e) \cong \mathrm{H}^i(\Hom_{A^e}(P_{\bullet}, A^e))$.
	Thus, $\Ext^i_{A^e}(M, A^e)$ is a left $\Lambda_1$-module.
	
	(2)
	Note that the canonical chain map $\Phi_{\bullet}: B^e \otimes_{A^e} \Hom_{A^e}(P_{\bullet}, A^e) \to \Hom_{A^e}(P_{\bullet}, B^e)$ is both $B^e$-linear and $H$-linear, as per \ref{mod-str-on-Hom(M, A^e)-lem} (3).
	Given that $M$ is perfect, $P_{\bullet}$ can be selected as a bounded complex of finitely generated projective $A^e$-modules.
	Hence, $\Phi_{\bullet}$ is a quasi-isomorphism.
	Taking homologies yields the desired result.
\end{proof}

\section{Homological determinats and Nakayama automotphisms for crossed products}\label{Nak-aut-cleft-ext-section}

Let $H$ be a Hopf algebra and $A$ be an algebra.
Suppose that $H$ {\it measures} $A$, that is, there is a $\kk$-linear map $H \otimes A \to A$, given by $h \otimes a \mapsto h \cdot a$, such that
$$h \cdot 1 = \vep(h)1, \;\; \text{ and } \;\; h \cdot (ab) = \sum (h_1 \cdot a)(h_2 \cdot b),$$
for all $h \in H$, $a, b \in A$.
Let $\sigma$ be a convolution invertible map in $\Hom(H \otimes H, A)$, meaning there exists $\sigma^{-1} \in \Hom(H \otimes H, A)$ satisfying $\sum \sigma(h_1, k_1) \sigma^{-1}(h_2, k_2) = \vep(hk)$ for all $h, k \in H$.
Let $A\#_{\sigma}H$ be the (in general nonassociative) algebra whose underlying vector space is $A \otimes H$ with multiplication
$$(a\#h)(b\#k) = \sum a(h_1 \cdot b)\sigma(h_2, k_1) \# h_3k_2$$
for all $a, b \in A$, $h, k \in H$.
The element $a \otimes h$ will usually be written $a \# h$.
The algebra $A\#_{\sigma}H$ is called {\it crossed product} of $A$ with $H$ if it is associative with $1\#1$ as identity element.
Recall from \cite{BCM1986} and \cite{DT1986} that $A\#_{\sigma}H$ is a crossed product if and only if the following two conditions are satisfied:
\begin{enumerate}
	\item $A$ is a twisted $H$-module. That is, $1 \cdot a = a$, for all $a \in A$, and
	\begin{equation}\label{h-k-a}
		h \cdot (k \cdot a) = \sum \sigma(h_1, k_1)(h_2k_2 \cdot a)\sigma^{-1}(h_3, k_3)
	\end{equation}
	for all $h, k \in H$, $a \in A$.
	\item $\sigma$ is normal cocycle. That is, $\sigma(h, 1) = \sigma(1, h) = \vep(h)1$, for all $h \in H$, and
	\begin{equation}\label{cocycle-eq}
		\sum \big( h_1 \cdot \sigma(k_1, l_1) \big) \sigma(h_2, k_2l_2) = \sum \sigma(h_1, k_1) \sigma(h_2k_2, l)
	\end{equation}
	for all $h, k, l \in H$.
\end{enumerate}

Deriving from \eqref{cocycle-eq}, we have the following identities for all $h, k, l \in H$:
\begin{eqnarray}\label{cocycle-eq-1}
	\sum \sigma^{-1}(h_1, k_1) \big( h_2 \cdot \sigma(k_2, l) \big) = \sum \sigma(h_1k_1, l_1) \sigma^{-1}(h_2, k_2l_2),
\end{eqnarray}
\begin{eqnarray}\label{cocycle-eq-3}
	\sum \sigma^{-1}(h_1k_1, l) \sigma^{-1}(h_2, k_2) = \sum \sigma^{-1}\big(h_1, k_1l_1 \big) \big( h_2 \cdot \sigma^{-1}(k_2, l_2) \big),
\end{eqnarray}

Clearly, each crossed product $A\#_{\sigma}H$ is a right $H$-comodule algebra such that $A \subseteq A\#_{\sigma}H$ is a faithfully flat $H$-Galois extension.

Let $B$ be a right $H$-comodule algebra. The extension $A(:= B^{co H}) \subset B$ is termed $H$-{\it cleft} if there exists a convolution invertible right $H$-comodule map $\gamma: H \to B$.

On one hand, consider the crossed product $A\#_{\sigma}H$.
Define the map $\gamma: H \to A\#_{\sigma}H$ by $\gamma(h) = 1\#h$, which is convolution invertible with the inverse given by
\begin{equation}\label{inverse-of-gamma}
	\gamma^{-1}(h) = \sum \sigma^{-1}(Sh_2, h_3) \# Sh_1,
\end{equation}
Thus, $A \subseteq A\#_{\sigma}H$ forms an $H$-cleft extension, as shown in Proposition 7.2.7 of \cite{Mon1993}.

On the other hand, every $H$-cleft extension $A \subseteq B$ is a crossed product of $A$ with $H$, where the action is defined by
\begin{equation}\label{H-cdot-A}
	h \cdot a = \sum \gamma(h_1) a \gamma^{-1}(h_2), \,\,\, \forall \, a \in A, h \in H,
\end{equation}
and the convolution invertible map $\sigma : H \otimes H \to A$ is given by
\begin{equation}\label{sigma-gamma}
	\sigma(h, k) = \sum \gamma(h_1)\gamma(k_1)\gamma^{-1}(h_2k_2), \,\,\, \forall \, h,k \in H,
\end{equation}
This result is established in Proposition 7.2.3 \cite{Mon1993}.

Given a cleft $H$-extension $A \subseteq B$ and a convolution invertible right $H$-comodule map $\gamma: H \to B$, for any $h \in H$, the following equality holds:
\begin{align*}
	\beta\big( \sum \gamma^{-1}(h_1) \otimes \gamma(h_2) \big) & = \beta\big( \sum \sigma^{-1}(Sh_2, h_3) \# Sh_1 \otimes 1 \# h_4 \big) \\
	& = \sum \sigma^{-1}(Sh_3, h_4) \sigma(Sh_2, h_5) \# (Sh_1) h_6 \otimes h_7 \\
	& = 1 \# h.
\end{align*}
Consequently, the translation map for the extension is given by
\begin{equation}\label{translation-map-for-cleft-ext}
	\kappa(h) = \sum \kappa^{1}(h) \otimes \kappa^2(h) = \sum \gamma^{-1}(h_1) \otimes \gamma(h_2). 
\end{equation}


\begin{lem}
	Let $B$ be a crossed product $A\#_{\sigma} H$.
	Then the algebra $\Lambda_i$, as defined in Definition \ref{Lambda_i-defn}, is identified with $\{ \sum x \# h_1 \otimes y \# S^{2i+1}h_2 \mid x, y \in A, h \in H \} \subseteq B^e$.
\end{lem}
\begin{proof}
	The inclusion $\{ \sum x \# h_1 \otimes y \# S^{2i+1}h_2 \mid x, y \in A, h \in H \} \subseteq \Lambda_i$ is evident. For any $x \# h \otimes y \# k \in \Lambda_i$, the definition of $\Lambda_i$ yields
	\[ \sum x \# h_1 \otimes y \# k_1 \otimes (S^{2i}h_2) k_2 = x \# h \otimes y \# k \otimes 1. \]
	Thus, we have
	\begin{align*}
		x \# h \otimes y \# k \otimes S^{2i+1}h_2 & = \sum x \# h_1 \otimes y \# k_1 \otimes (S^{2i+1}h_2) (S^{2i}h_3) k_2 \\
		& = \sum x \# h \otimes y \# k_1 \otimes k_2.
	\end{align*}
	Thus, $x \# h \otimes y \# k = x \# h \otimes y \# \vep(k_1) k_2 = x \# h \otimes y \# \vep(k) S^{2i+1}h_2$.
	This implies $\Lambda_i \subseteq \{ \sum x \# h_1 \otimes y \# S^{2i+1}h_2 \mid x, y \in A, h \in H \}$.
	Therefore, the result follows.
\end{proof}

Smash products are a specialized form of cleft extensions. Consider a right $H$-module algebra $A$ with action denoted by ``$\cdot$" and a cocycle $\sigma(h, k) = \vep(h)\vep(k)$ for all $h, k \in H$.
In this setting, the crossed product $A \#_{\sigma} H$ constitutes a smash product, denoted $A\#H$.
Let $\Delta_i$ be the algebra whose underlying vector space is $A^e \otimes H$, and whose multiplication is defined by
\[ (x \otimes y \otimes h)(x' \otimes y' \otimes h') = \sum x(h_1 \cdot x') \otimes (S^{2i}h_3 \cdot y') y \otimes h_2h'. \]
The algebra $\Delta_i$, introduced in \cite{Kay2007} and \cite{ LMeu2019}, has been utilized to examine the Calabi--Yau property of smash products, as detailed in \cite{LWZ2012}, \cite{RRZ2014}, \cite{LMeu2019}.

Subsequently, we will show that $\Lambda_i$ is essentially $\Delta_i$ in the context of smash products.

\begin{lem}
	The map $F: \Delta_{i} \to \Lambda_i$ defined by
	\[ x \otimes y \otimes h \mapsto \sum x \# h_1 \otimes S^{2i+1}h_3 \cdot y \# S^{2i+1}h_2 \]
	is an algebra isomorphism.
\end{lem}
\begin{proof}
	The map $F$ is evidently bijective, with the inverse map given by
	\[ \sum x \# h_1 \otimes y \# S^{2i+1}h_2 \mapsto \sum x \otimes (S^{2i}h_2 \cdot y) \otimes h_1. \]
	To show that $F$ is an algebra isomorphism, we verify that it preserves the algebra structure. We have
	\begin{align*}
		& F\big( (x \otimes y \otimes h)(x' \otimes y' \otimes h') \big) \\
		& = \varphi\big( \sum x(h_1 \cdot x') \otimes (S^{2i}h_3 \cdot y') y \otimes h_2h' \big) \\
		& = \sum x(h_1 \cdot x') \# h_2h'_1 \otimes S^{2i+1}(h_4{h'}_3) \cdot \big( (S^{2i}h_5 \cdot y') y \big) \# S^{2i+1}(h_3{h'}_2) \\
		& = \sum x (h_1 \cdot {x'}) \# h_2{h'}_1 \otimes (S^{2i+1}{h'}_4 \cdot {y'}) \big( S^{2i+1}(h_4{h'}_3) \cdot y \big) \# S^{2i+1}(h_3{h'}_2) \\
		& = \sum x (h_1 \cdot {x'}) \# h_2{h'}_1 \otimes (S^{2i+1}{h'}_4 \cdot {y'}) \big( (S^{2i+1}{h'}_3) (S^{2i+1}h_4) \cdot y \big) \# (S^{2i+1}{h'}_2)(S^{2i+1}h_3) \\
		& = (\sum x \# h_1 \otimes S^{2i+1}h_3 \cdot y \# S^{2i+1}h_2) (\sum {x'} \# {h'}_1 \otimes S^{2i+1}{h'}_3 \cdot {y'} \# S^{2i+1}{h'}_2 ) \\
		& = F(x \otimes y \otimes h) F(x' \otimes y' \otimes h'),
	\end{align*}
	which demonstrates that $F$ is an algebra isomorphism.
\end{proof}

Subsequently, we consider $B$ as a crossed product $A\#_{\sigma}H$ with the comodule map $\gamma: H \to A$ defined by $\gamma(h) = 1\#h$.

\begin{lem}
	For any $h, k \in H$,
	\begin{equation}\label{gamma^{-1}-gamma^{-1}}
		\gamma^{-1}(k) \gamma^{-1}(h) = \gamma^{-1}(h_1k_1) \big( \sigma^{-1}(h_2, k_2) \# 1 \big).
	\end{equation}
\end{lem}
\begin{proof}
	We proceed with the calculation as follows: 
	\begin{align*}
		& \gamma^{-1}(k) \gamma^{-1}(h) \\
		& \xlongequal{\eqref{inverse-of-gamma}} \sum \big( \sigma^{-1}(Sk_2, k_3) \# Sk_1 \big) \big( \sigma^{-1}(Sh_2, h_3) \# Sh_1 \big) \\
		& = \sum \sigma^{-1}(Sk_4, k_5) \big( Sk_3 \cdot \sigma^{-1}(Sh_3, h_4) \big) \sigma(Sk_2, Sh_2) \# (Sk_1) (Sh_1) \\
		& = \sum \sigma^{-1}(Sk_5, k_6) \sigma^{-1}\big( Sk_4, (Sh_4) h_5 \big) \big( Sk_3 \cdot \sigma^{-1}(Sh_3, h_6) \big) \sigma(Sk_2, Sh_2) \# S(h_1k_1) \\
		& \xlongequal{\eqref{cocycle-eq-3}} \sum \sigma^{-1}(Sk_5, k_6) \sigma^{-1}\big( (Sk_4) (Sh_4), h_5 \big) \sigma^{-1}(Sk_3, Sh_3) \sigma(Sk_2, Sh_2) \# S(h_1k_1) \\
		& = \sum \sigma^{-1}\big( S(h_3k_3) h_4, k_4 \big) \sigma^{-1}\big( S(h_2k_2), h_5 \big) \# S(h_1k_1) \\
		& \xlongequal{\eqref{cocycle-eq-3}} \sum \sigma^{-1}\big( S(h_3k_3), h_4k_4 \big) \big( S(h_2k_2) \cdot \sigma^{-1}(h_5, k_5) \big) \# S(h_1k_1) \\
		& = \sum \Big( \sigma^{-1}\big( S(h_2k_2), h_3k_3 \big) \# S(h_1k_1) \Big) \big( \sigma^{-1}(h_4, k_4) \# 1 \big) \\
		& = \gamma^{-1}(h_1k_1) \big( \sigma^{-1}(h_2, k_2) \# 1 \big).
	\end{align*}
	This completes the proof.
\end{proof}

In general, $H$ cannot inherently be a subalgebra $\Lambda_i$ unless $\sigma$ is trivial.
However, there exists an embedding map $\pi_i: H \to \Lambda_i$ defined by $h \mapsto \sum \gamma(h_1) \otimes \gamma^{-1}(S^{2i}h_2)$.

The following properties hold for this map.

\begin{lem}\label{pi-lemma}
	For any $h, k \in H$, in $\Lambda_i$,
	\begin{enumerate}
		\item[(1)] $\pi_i(1) = 1 \otimes 1$,
		\item[(2)] $\pi_i(h) (a \# 1 \otimes a' \# 1) = \sum (h_1 \cdot a \# 1 \otimes S^{2i}h_3 \cdot a' \# 1) \pi_i(h_2)$,
		\item[(3)] $\pi_i(h)\pi_i(k) = \sum \big( \sigma(h_1, k_1) \# 1 \otimes \sigma^{-1}(S^{2i}h_3, S^{2i}k_3) \# 1 \big) \pi_i(h_2k_2)$.
	\end{enumerate}
\end{lem}
\begin{proof}
	(1) This follows from $\gamma(1_H) = 1_B$.
	
	(2)
	 We compute
	\begin{align*}
		\pi_i(h) (a \otimes a') = & \big( \sum \gamma(h_1) \otimes \gamma^{-1}(S^{2i}h_2) \big) (a \otimes a') \\
		= & \sum \gamma(h_1) a \otimes a' \gamma^{-1}(S^{2i}h_2) \\
		= & \sum \gamma(h_1) a \gamma^{-1}(h_2) \gamma(h_3) \otimes \gamma^{-1}(S^{2i}h_4) \gamma(S^{2i}h_5) a' \gamma^{-1}(S^{2i}h_6) \\
		= & \sum \big( \gamma(h_1) a \gamma^{-1}(h_2) \otimes \gamma(S^{2i}h_5) a' \gamma^{-1}(S^{2i}h_6) \big) \big( \gamma(h_3) \otimes \gamma^{-1}(S^{2i}h_4) \big) \\
		= & \sum (h_1 \cdot a \otimes S^{2i}h_3 \cdot a') \pi_i(h_2).
	\end{align*}
	
	(3)
	Using $\gamma(h)\gamma(k) = (1 \# h)(1 \# k) = \sum \sigma(h_1, k_1) \gamma(h_2k_2)$, and 
	\eqref{gamma^{-1}-gamma^{-1}}, we get
	\begin{align*}
		\pi_i(h) \pi_i(k) & = \big( \sum \gamma(h_1) \otimes \gamma^{-1}(S^{2i}h_2) \big) \big( \sum \gamma(k_1) \otimes \gamma^{-1}(S^{2i}k_2) \big) \\
		& = \sum \gamma(h_1) \gamma(k_1) \otimes \gamma^{-1}(S^{2i}k_2) \gamma^{-1}(S^{2i}h_2) \\
		& \xlongequal{\eqref{gamma^{-1}-gamma^{-1}}} \sum \big( \sigma(h_1, k_1) \# 1\big) \gamma(h_2k_2) \otimes \gamma^{-1}(S^{2i}(h_3k_3)) \big( \sigma^{-1}(S^{2i}h_4, S^{2i}k_4) \# 1 \big) \\
		 & =\sum \big( \sigma(h_1, k_1) \# 1 \otimes \sigma^{-1}(S^{2i}h_3, S^{2i}k_3) \# 1 \big) \pi_i(h_2k_2).
	\end{align*}
	The proof is completed.
\end{proof}

For any $M \in {_{\Lambda_1}\!\M}$, $h \in H$ and $m \in M$, we define
\begin{equation}\label{h-cdot-m-defn}
	h \cdot m := \pi_1(h) m = \big( \sum \gamma(h_1) \otimes \gamma^{-1}(S^2h_2) \big) \vtr m.
\end{equation}

Observe that $M$ is generally not an $H$-module under the action defined above.
Building upon Lemma \ref{pi-lemma}, we establish the following:

\begin{lem}\label{equi-bimod}
	For all $h, k \in H$, $a, a' \in A$, and $m \in M$,
	\begin{equation}\label{1_H-cdot-u}
		1_H \cdot m = m,
	\end{equation}
	\begin{equation}\label{h->aua'}
		h \cdot (ama') = \sum (h_1 \cdot a) (h_2 \cdot m) (S^2h_2 \cdot a'),
	\end{equation}
	\begin{equation}\label{h-cdot-k-cdot-u}
		h \cdot (k \cdot m) = \sum \sigma(h_1, k_1) (h_2k_2 \cdot m) \sigma^{-1}(S^2h_3, S^2k_3).
	\end{equation}
\end{lem}

\subsection{Homological determinant}

Inspired by the Nakayama automorphism’s expression for smash products of skew Calabi--Yau algebras via homological determinants, we introduce the concept of homological determinants for crossed products.

Let $A$ be a skew Calabi--Yau algebra of dimension $d$ with a Nakayama automorphism $\mu_A$.
Hence there exists a generator $e$ of the $A^e$-module $\Ext^d_{A^e}(A, A^e)(\cong A_{\mu_A})$ such that $ea = \mu_A(a)e$ for any $a \in A$.
The generator $e$ will be referred to a {\it $\mu_A$-twisted volume} on $A$.
As established in Proposition \ref{B^e-otimes_{A^e}-Ext^i_{A^e}(M, A^e)-prop}, $\Ext^d_{A^e}(A, A^e)$ is a $\Lambda_1$-module. Since $e$ also generates $\Ext^d_{A^e}(A, A^e)$ as a free left $A$-module of rank one, there exists a $\kk$-linear map $\varphi: H \to A$ defined by
\begin{equation}\label{varphi-defn}
	h \cdot e = \pi_1(h) e = \varphi(h)e.
\end{equation}

Then we can deduce the subsequent outcome from Lemma \ref{equi-bimod}.

\begin{lem}
	For any $h, k \in H$, $a \in A$, then
	\begin{equation}\label{varphi-eq0}
		\varphi(1_H) = 1_A,
	\end{equation}
	\begin{equation}\label{varphi-eq1}
		\sum \big( h_1 \cdot \mu_A(a) \big) \varphi(h_2) = \sum \varphi(h_1) \mu_A(S^2h_2 \cdot a),
	\end{equation}
	\begin{equation}\label{varphi-eq2}
		\sum \big( h_1 \cdot \varphi(k) \big)\varphi(h_2) = \sum \sigma(h_1, k_1) \varphi(h_2k_2) \mu_A\big( \sigma^{-1}(S^2h_3, S^2k_3) \big).
	\end{equation}
\end{lem}
\begin{proof}
	(1) Note that $e$ is a generator of a free left $A$-module $\Ext^d_{A^e}(A, A^e)$.
	We have $\varphi(1_H) = 1_A$ because $e = 1_H \cdot e = \varphi(1_H) e$ by \eqref{1_H-cdot-u}.
	
	(2) We compute 
	\begin{align*}
		\sum \big( h_1 \cdot \mu_A(a) \big) \varphi(h_2) e & \xlongequal{\eqref{varphi-defn}} \sum \big( h_1 \cdot \mu_A(a) \big) (h_2 \cdot e) \\
		& \xlongequal{\eqref{h->aua'}} h \cdot \big( \mu_A(a) e \big) = h \cdot (e a) \\
		& \xlongequal{\eqref{h->aua'}} \sum ( h_1 \cdot e ) (S^2h_2 \cdot a) \\
		& \xlongequal{\eqref{varphi-defn}} \sum \varphi(h_1) e (S^2h_2 \cdot a) \\
		& = \sum \varphi(h_1) \mu_A(S^2h_2 \cdot a) e
	\end{align*}
	implies that \eqref{varphi-eq1} holds.
	
	(3) We calculate 
	\begin{align*}
		\sum \big( h_1 \cdot \varphi(k) \big) \varphi(h_2) e & \xlongequal{\eqref{varphi-defn}} \sum \big( h_1 \cdot \varphi(k) \big) (h_2 \cdot e) \\
		& \xlongequal{\eqref{h->aua'}} h \cdot (\varphi(k) e) \xlongequal{\eqref{varphi-defn}} h \cdot (k \cdot e) \\
		& \xlongequal[]{\eqref{h-cdot-k-cdot-u}} \sum \sigma(h_1, k_1) \big( (h_2k_2) \cdot e \big) \sigma^{-1}(S^2h_3, S^2k_3) \\
		& \xlongequal{\eqref{varphi-defn}} \sum \sigma(h_1, k_1)\varphi(h_2k_2) \mu_A\big( \sigma^{-1}(S^2h_3, S^2k_3) \big) e
	\end{align*}
	showing that \eqref{varphi-eq2} is valid.
	This completes the proof.
\end{proof}

We now demonstrate that $\varphi$ is convolution invertible, with its inverse, referred to as the homological determinant of the $H$-action on $A$.

\begin{lem}\label{varphi-*-invertible}
	The map $\varphi$ defined by \eqref{varphi-defn} is convolution invertible with the convolution inverse $\psi$ which is defined by
	\begin{equation}\label{Hdet-definition}
		h \mapsto \sum \sigma^{-1}(h_{5}, S^{-1}h_4) h_{6}\cdot \Big( \varphi(S^{-1}h_3) \mu_A\big( \sigma(Sh_2, S^2h_1) \big) \Big).
	\end{equation}
\end{lem}
\begin{proof}
	For the convolution product $(\psi * \varphi)(h)$, we have
	\begin{align*}
		& = \sum \psi(h_1) \varphi(h_2) \\
		& = \sum \sigma^{-1}(h_{5}, S^{-1}h_4) h_{6} \cdot \Big( \varphi(S^{-1}h_3) \mu_A\big( \sigma(Sh_2, S^2h_1) \big) \Big) \varphi(h_7) \\
		& = \sum \sigma^{-1}(h_{5}, S^{-1}h_4) \big( h_{6} \cdot \varphi(S^{-1}h_3) \big) \Big( h_7 \cdot  \mu_A \big( \sigma(Sh_2, S^2h_1) \big) \Big) \varphi(h_8) \\
		& \xlongequal[]{\eqref{varphi-eq1}} \sum \sigma^{-1}(h_{5}, S^{-1}h_4) \big( h_{6} \cdot \varphi(S^{-1}h_3) \big) \varphi(h_7) \mu_A\big( S^2h_8 \cdot \sigma(Sh_2, S^2h_1) \big) \\
		& \xlongequal[]{\eqref{varphi-eq2}} \sum \sigma^{-1}(h_{7}, S^{-1}h_6) \sigma(h_8, S^{-1}h_5) \varphi\big( h_9(S^{-1}h_4) \big) \mu_A\big( \sigma^{-1}(S^2h_{10}, Sh_3) \big) \\
		& \qquad \qquad \mu_A\big( S^2h_{11} \cdot \sigma(Sh_2, S^2h_1) \big) \\
		& = \sum \varphi(1_H) \mu_A\big( \sigma^{-1}(S^2h_{4}, Sh_3) \big) \mu_A\big( S^2h_{5} \cdot \sigma(Sh_2, S^2h_1) \big) \\
		& \xlongequal[]{\eqref{varphi-eq0}} \sum \mu_A\Big( \sigma^{-1}(S^2h_{4}, Sh_3) \big( S^2h_{5} \cdot \sigma(Sh_2, S^2h_1) \big) \Big) \\
		& \xlongequal[]{\eqref{cocycle-eq-1}} \sum \mu_A \Big(\sigma\big( (S^2h_5)(Sh_4), S^2h_1 \big) \sigma^{-1}\big( S^2h_6, (Sh_3)(S^2h_2) \big) \Big) \\
		& = \sum \mu_A \big(\sigma(1, S^2h_1) \sigma^{-1}( S^2h_2, 1) \big) \\
		& = \vep(h).
	\end{align*}
	Analogously, for $(\varphi * \psi)(h)$
	\begin{align*}
		& = \sum \varphi(h_1) \psi(h_2) \\
		& = \sum \varphi(h_1) \sigma^{-1}(h_{6}, S^{-1}h_5) h_{7} \cdot \Big( \varphi(S^{-1}h_4) \mu_A\big( \sigma(Sh_3, S^2h_2) \big) \Big) \\
		& = \sum \sigma^{-1}(h_9, S^{-1}h_8) \sigma(h_{10}, S^{-1}h_7) \big( h_{11}S^{-1}h_6 \cdot \varphi(h_1) \big) \sigma^{-1}(h_{12}, S^{-1}h_5) \\
		& \qquad h_{13} \cdot \Big( \varphi(S^{-1}h_4) \mu_A\big( \sigma(Sh_3, S^2h_2) \big) \Big) \\
		& \xlongequal[]{\eqref{h-k-a}} \sum \sigma^{-1}(h_{7}, S^{-1}h_6)  h_{8} \cdot \big( S^{-1}h_5 \cdot \varphi(h_1) \big) h_{9} \cdot \Big( \varphi(S^{-1}h_4) \mu_A\big( \sigma(Sh_3, S^2h_2) \big) \Big) \\
		& = \sum \sigma^{-1}(h_{7}, S^{-1}h_6)  h_{8} \cdot \Big( \big( S^{-1}h_5 \cdot \varphi(h_1) \big) \varphi(S^{-1}h_4) \mu_A\big( \sigma(Sh_3, S^2h_2) \big) \Big) \\
		& \xlongequal[]{\eqref{varphi-eq2}} \sum \sigma^{-1}(h_{10}, S^{-1}h_9) \\
		& \qquad \quad h_{11} \cdot \Big( \big( \sigma(S^{-1}h_8, h_1) \varphi\big( (S^{-1}h_7) h_2 \big) \mu_A\big( \sigma^{-1}(Sh_6, S^{2}h_3) \big) \mu_A \big( \sigma(Sh_5, S^2h_4) \big) \Big) \\
		& = \sum \sigma^{-1}(h_{4}, S^{-1}h_3) h_{5} \cdot \big( \sigma(S^{-1}h_2, h_1) \varphi(1_H) \big) \\
		& \xlongequal{\eqref{varphi-eq0}} \sum \sigma^{-1}(h_{4}, S^{-1}h_3)  h_{5} \cdot \sigma(S^{-1}h_2, h_1) \\
		& \xlongequal[]{\eqref{cocycle-eq-1}} \sum \sigma\big( h_{5}(S^{-1}h_4), h_1 \big) \sigma^{-1}\big( h_6, (S^{-1}h_3) h_2 \big) \\
		& = \sum \sigma(1, h_1) \sigma^{-1}(h_2, 1) \\
		& = \vep(h).
	\end{align*}
	Thus, $\psi$ is the convolution inverse of $\varphi$.
\end{proof}

\begin{defn}\label{Hdet-defn}
	The $\kk$-linear map $\psi: H \to A$ from Lemma \ref{varphi-*-invertible} is termed the {\it homological determinant} of the $H$-action on $A=B^{coH}$ with respect to the $\mu_A$-twisted volume $e$, and is denoted by $\Hdet$.
\end{defn}

Given a crossed product $B:= A\#_{\sigma} H$ of $A$ with $H$, if $A$ is a right $H$-module algebra under the action denoted by ``$\cdot$", then the homological determinant defined as in Definition \ref{Hdet-defn}, coincides with that described in \cite{KKZ2009, LMeu2019}.
For instance, $A$ is always an $H$-module algebra when $H$ is cocommutative and $\sigma(H \otimes H) \subseteq Z(A)$, where $Z(A)$ is the center of $A$.

\subsection{Nakayama automorphisms of crossed products}

In this subsection, we will always assume that $B$ is a crossed product $A\#_{\sigma}H$ with the comodule map $\gamma: H \to A$ defined by $\gamma(h) = 1\#h$.

\begin{lem}
	For any $M \in {_{\Lambda_1}\!\M}$, $B \otimes_A M \otimes H$ is a Hopf bimodule in ${_{B^e}\!\M^{H^e}_{H}}$, where the right $H$-module structure on $B \otimes_A M \otimes H$ is given by
	\begin{equation}\label{H-mod-on-B-otimes-M-otimes-H}
		(b \otimes_A m \otimes h) \leftharpoondown k = b \otimes_A m \otimes (Sk)h,
	\end{equation}
	the left $B^e$-module structure on $B \otimes_A M \otimes H$ is given by
	\begin{equation}\label{BMH-B^e-mod}
		(x \otimes y) \big( b \otimes_{A} m \otimes h \big) = \sum xb \gamma^{-1}(S^{-3}y_3) \otimes_A (S^{-3}y_2 \cdot m) \gamma(S^{-1}y_1)y_0 \otimes hy_4,
	\end{equation}
	the $H^e$-comodule structure on $B \otimes_A M \otimes H$ is given by
	\begin{equation}\label{BMH-H^e-comod}
		\rho(b \otimes_A m \otimes h) = \sum b_0 \otimes_A m \otimes h_2 \otimes b_1(S^{-1}h_1) \otimes h_3
	\end{equation}
	for any $b, x, y \in B$, $h, k \in H$ and $m \in M$.
\end{lem}
\begin{proof}
	The proof is straightforward and thus omitted.
\end{proof}

Recall that for any $M \in {_{\Lambda_1}\!\M}$, by Theorem \ref{HG-bimod-str-thm-1}, $B^e \otimes_{A^e} M$ is a Hopf bimodule in ${_{B^e}\!\M^{H^e}_{H}}$, and the right $H$-module structure is given by
\begin{equation}\label{H-mod-on-B^e-otimes-M-1}
	\begin{split}
		& (x \otimes y \otimes_{A^e} m) h \\
		& \xlongequal{\eqref{H-mod-on-B^e-otimes-M}} \sum x \kappa^1(S^{-1}h_2) \otimes \kappa^2(Sh_1)y \otimes_{A^e} \big( \kappa^2(S^{-1}h_2) \otimes \kappa^1(Sh_1) \big) \vtr m \\
		& \xlongequal{\eqref{translation-map-for-cleft-ext}} \sum x \gamma^{-1}(S^{-1}h_4) \otimes \gamma(Sh_1)y \otimes_{A^e} \big( \gamma(S^{-1}h_3) \otimes \gamma^{-1}(Sh_2) \big) \vtr m \\
		& \xlongequal{\eqref{h-cdot-m-defn}} \sum x \gamma^{-1}(S^{-1}h_3) \otimes \gamma(Sh_1)y \otimes_{A^e} S^{-1}h_2 \cdot m
	\end{split}
\end{equation}
for any $x, y \in B$, $h \in H$ and $m \in M$.

\begin{lem}\label{omega_B=Botimes_AU}
	For any $M \in {_{\Lambda_1}\!\M}$, there is an isomorphism
	\[ F: B^e \otimes_{A^e} M \stackrel{\cong}{\To} B \otimes_A M \otimes H \]
	in ${_{B^e}\!\M^{H^e}_{H}}$, which is defined by 
	\begin{equation}
		F(x \otimes y \otimes_{A^e} m) = \sum x\gamma^{-1}(S^{-3}y_3) \otimes_A (S^{-3}y_2 \cdot m) \gamma(S^{-1}y_1)y_0 \otimes y_4.
	\end{equation}
	where $x, y \in B$, and $m \in M$.
\end{lem}
\begin{proof}
	Note that $M$ is an $A$-$A$-bimodule since $A^e$ is a subalgebra of $\Lambda_1$.
	For any $b := a \# h \in B$,
	\begin{align*}
		\sum \gamma(S^{-1}b_1) b_0 & = \sum (1 \# S^{-1}h_2) (a \# h_1) = \sum (S^{-1}h_5 \cdot a) \sigma(S^{-1}h_4, h_1) \# (S^{-1}h_3)h_2 \\
		& = \sum (S^{-1}h_3 \cdot a) \sigma(S^{-1}h_2, h_1) \# 1 \in A.
	\end{align*}
	Thus, for any $x, y \in B$, and $m \in M$, the expression
	\[ \sum x\gamma^{-1}(S^{-3}y_3) \otimes_A (S^{-3}y_2 \cdot m) \gamma(S^{-1}y_1)y_0 \otimes y_4 \]
	is unambiguous.
	For any $a, a \in A$, we have
	\begin{align*}
		& F(xa \otimes a'y \otimes_{A^e} m) \\
		& = \sum xa\gamma^{-1}(S^{-3}y_3) \otimes_A (S^{-3}y_2 \cdot m) \gamma(S^{-1}y_1)a'y_0 \otimes y_4 \\
		& = \sum x\gamma^{-1}(S^{-3}y_7) \gamma^{-1}(S^{-3}y_6)a\gamma^{-1}(S^{-3}y_5) \otimes_A \\
		& \qquad (S^{-3}y_4 \cdot m) \gamma(S^{-1}y_3)a'\gamma^{-1}(S^{-1}y_2)\gamma(S^{-1}y_1)y_0 \otimes y_8 \\
		& \xlongequal{\eqref{H-cdot-A}} \sum x\gamma^{-1}(S^{-3}y_5) (S^{-3}y_4 \cdot a) \otimes_A (S^{-3}y_3 \cdot m) (S^{-1}y_2 \cdot a')\gamma(S^{-1}y_1)y_0 \otimes y_6 \\
		& = \sum x\gamma^{-1}(S^{-3}y_5) \otimes_A (S^{-3}y_4 \cdot a) (S^{-3}y_3 \cdot m) (S^{-1}y_2 \cdot a')\gamma(S^{-1}y_1)y_0 \otimes y_6\\
		& \xlongequal{\eqref{h->aua'}} \sum x\gamma^{-1}(S^{-3}y_3) \otimes_A \big( S^{-3}y_2 \cdot (ama') \big) \gamma(S^{-1}y_1)y_0 \otimes y_4 \\
		& = F(x \otimes y \otimes_{A^e} ama').
	\end{align*}
	Therefore, $F$ is well defined.
	
	Next, we verify that $F$ is a morphism in ${_{B^e}\!\M^{H^e}_{H}}$.
	To show that $F$ is a right $H$-module morphism, we calculate
	\begin{align*}
		& F\big( (x \otimes y \otimes_{A^e} m) h \big) \\
		& \xlongequal{\eqref{H-mod-on-B^e-otimes-M-1}} F\big( \sum x \gamma^{-1}(S^{-1}h_3) \otimes \gamma(Sh_1) y \otimes_{A^e} S^{-1}h_2 \cdot m \big) \\
		& = \sum x \gamma^{-1}(S^{-1}h_7) \gamma^{-1}\Big( S^{-3}\big( (Sh_{2}) y_3 \big) \Big) \otimes_A S^{-3}\big( (Sh_{3}) y_2 \big) \cdot (S^{-1}h_6 \cdot m) \\
		& \qquad \gamma\Big( S^{-1}\big( (Sh_{4}) y_1 \big) \Big) \gamma(Sh_5) y_0 \otimes (Sh_1) y_4 \\
		& = \sum x \gamma^{-1}(S^{-1}h_7) \gamma^{-1}\big( (S^{-3}y_3) (S^{-2}h_{2}) \big) \otimes_A \big( (S^{-3}y_2) (S^{-2}h_{3}) \big) \cdot (S^{-1}h_6 \cdot m) \\
		& \qquad \gamma\big( (S^{-1}y_1) h_{4} \big) \gamma(Sh_5) y_0 \otimes (Sh_1) y_4 \\
		& \xlongequal{\eqref{h-cdot-k-cdot-u}} \sum x \gamma^{-1}(S^{-1}h_{11}) \gamma^{-1}\big( (S^{-3}y_5) (S^{-2}h_{2}) \big) \otimes_A \sigma\big( (S^{-3}y_4) (S^{-2}h_{3}), S^{-1}h_{10} \big)  \\
		& \qquad \quad \big( (S^{-3}y_3) (S^{-2}h_{4}) (S^{-1}h_9)  \cdot m \big) \sigma^{-1}\big( (S^{-1}y_2) h_{5}, Sh_{8} \big) \gamma\big( (S^{-1}y_1) h_{6} \big) \gamma(Sh_7) y_0 \\
		& \qquad \quad \otimes (Sh_1) y_6 \\
		& \xlongequal{\eqref{gamma^{-1}-gamma^{-1}}} \sum x \gamma^{-1}\big( (S^{-3}y_3) (S^{-2}h_{2}) (S^{-1}h_{7})\big) \otimes_A \big( (S^{-3}y_2) (S^{-2}h_{3}) (S^{-1}h_6) \cdot m \big) \\
		& \qquad \quad \gamma\big( (S^{-1}y_1) h_{4} (Sh_{5}) \big) y_0 \otimes (Sh_1) y_4 \\
		& = \sum x\gamma^{-1}(S^{-3}y_3) \otimes_A (S^{-3}y_2 \cdot m) \gamma(S^{-1}y_1)y_0 \otimes (Sh) y_4 \\
		& \xlongequal{\eqref{H-mod-on-B-otimes-M-otimes-H}} F(x \otimes y \otimes_{A^e} m) \leftharpoondown h.
	\end{align*}
	
	To show that $F$ is a left $B^e$-module morphism, we compute
	\begin{align*}
		& F\big( (b' \otimes b) (x \otimes y \otimes_{A^e} m) \big) \\
		& = F\big( (b'x \otimes yb \otimes_{A^e} m) \big) \\
		& = \sum b'x \gamma^{-1}(S^{-3}(y_3b_3)) \otimes_A (S^{-3}(y_2b_2) \cdot m) \gamma(S^{-1}(y_1b_1)) y_0 b_0 \otimes y_4b_4 \\
		& \xlongequal{\eqref{gamma^{-1}-gamma^{-1}}} \sum b'x \gamma^{-1}(S^{-3}y_5) \gamma^{-1}(S^{-3}b_5) \sigma(S^{-3}b_4, S^{-3}y_4) \otimes_A (S^{-3}(y_3b_3) \cdot m) \\
		& \qquad \quad \sigma^{-1}(S^{-1}b_2, S^{-1}y_2) \gamma(S^{-1}b_1) \gamma(S^{-1}y_1) y_0b_0 \otimes y_6b_6 \\
		& \xlongequal{\eqref{h-cdot-k-cdot-u}} \sum b'x \gamma^{-1}(S^{-3}y_3)\gamma^{-1}(S^{-3}b_5) \otimes_A \big( S^{-3}b_4 \cdot (S^{-3}y_2 \cdot m) \big) \\
		& \qquad \quad \gamma(S^{-1}b_3) \gamma(S^{-1}y_1) y_0 \gamma^{-1}(S^{-1}b_2) \gamma(S^{-1}b_1) b_0 \otimes y_4b_6 \\
		& \xlongequal{\eqref{h->aua'}} \sum b'x \gamma^{-1}(S^{-3}y_3)\gamma^{-1}(S^{-3}b_3) \otimes_A S^{-3}b_2 \cdot \big( (S^{-3}y_2 \cdot m) \gamma(S^{-1}y_1) y_0 \big) \\
		& \qquad \quad \gamma(S^{-1}b_1)  b_0 \otimes y_4b_4 \\
		& \xlongequal{\eqref{BMH-B^e-mod}} (b' \otimes b) F(x \otimes y \otimes_{A^e} m).
	\end{align*}
	
	To show that $F$ is an $H^e$-comodule morphism, we calculate
	\begin{align*}
		& (F \otimes \id_{H^e}) \rho(x \otimes y \otimes_{A^e} m) \\
		& = F(\sum x_0 \otimes y_0 \otimes_{A^e} m) \otimes x_1 \otimes y_1 \\
		& = \sum x_0 \gamma^{-1}(S^{-3}y_3) \otimes_A (S^{-3}y_2 \cdot m) \gamma(S^{-1}y_1)y_0 \otimes y_4 \otimes x_1 \otimes y_5 \\
		& = \sum x_0 \gamma^{-1}(S^{-3}y_3) \otimes_A (S^{-3}y_2 \cdot m) \gamma(S^{-1}y_1)y_0 \otimes y_6 \otimes x_1 (S^{-2}y_4) (S^{-1}y_5) \otimes y_7 \\
		& \xlongequal{\eqref{BMH-H^e-comod}} \rho \big( \sum x_0 \gamma^{-1}(S^{-3}y_3) \otimes_A (S^{-3}y_2 \cdot m) \gamma(S^{-1}y_1)y_0 \otimes y_4 \big) \\
		& = \rho F(x \otimes y \otimes_{A^e} m).
	\end{align*}
	This implies that $F$ is a morphism in ${_{B^e}\!\M^{H^e}_{H}}$.
	
	A well-defined map $G: B \otimes_A M \otimes H \to B^e \otimes_{A^e} M$ is given by
	\[ b \otimes_A m \otimes h \mapsto \sum b \gamma^{-1}(S^{-2}h_1) \otimes \gamma(h_3) \otimes_{A^e} S^{-2}h_2 \cdot m. \]
	We demonstrate that $G$ is the inverse of $F$ by computing
	\begin{align*}
		& FG(b \otimes_A m \otimes h) \\
		& = F\big( \sum b \gamma^{-1}(S^{-2}h_1) \otimes \gamma(h_3) \otimes_{A^e} S^{-2}h_2 \cdot m \big) \\
		& = \sum b \gamma^{-1}(S^{-2}h_1) \gamma^{-1}(S^{-3}h_6) \otimes_A \big( S^{-3}h_5 \cdot (S^{-2}h_2 \cdot m) \big) \gamma(S^{-1}h_4) \gamma(h_3) \otimes h_7 \\
		& \xlongequal{\eqref{gamma^{-1}-gamma^{-1}}} \sum b \gamma^{-1}\big( (S^{-3}h_{10}) (S^{-2}h_1) \big) \sigma(S^{-3}h_9, S^{-2}h_2) \otimes_A \big( S^{-3}h_8 \cdot (S^{-2}h_3 \cdot m) \big) \\
		& \qquad \quad  \sigma(S^{-1}h_7, h_4) \gamma\big( (S^{-1}h_6) h_5 \big) \otimes h_{11} \\
		& = \sum b \gamma^{-1}\big( (S^{-3}h_{8}) (S^{-2}h_1) \big) \otimes_A \sigma(S^{-3}h_7, S^{-2}h_2) \big( S^{-3}h_6 \cdot (S^{-2}h_3 \cdot m) \big) \\
		& \qquad \quad \sigma(S^{-1}h_5, h_4) \otimes h_{9} \\
		& \xlongequal{\eqref{h-cdot-k-cdot-u}} \sum b \gamma^{-1}\big( (S^{-3}h_{4}) (S^{-2}h_1) \big) \otimes_A \big( (S^{-3}h_3) (S^{-2}h_2)  \big) \cdot m \otimes h_{5} \\
		& = b \otimes_A (1_H \cdot m) \otimes h \\
		& \xlongequal{\eqref{1_H-cdot-u}} b \otimes_A m \otimes h,
	\end{align*}
	and
	\begin{align*}
		& GF(x \otimes y \otimes_{A^e} m) \\
		& = G\big( \sum x\gamma^{-1}(S^{-3}y_3) \otimes_A (S^{-3}y_2 \cdot m) \gamma(S^{-1}y_1)y_0 \otimes y_4 \big) \\
		& = \sum x\gamma^{-1}(S^{-3}y_3) \gamma^{-1}(S^{-2}y_4) \otimes \gamma(y_6) \otimes_{A^e} S^{-2}y_5 \cdot \big( (S^{-3}y_2 \cdot m) \gamma(S^{-1}y_1)y_0 \big) \\
		& \xlongequal{\eqref{h->aua'}} \sum x\gamma^{-1}(S^{-3}y_3) \gamma^{-1}(S^{-2}y_4) \otimes \gamma(y_7) \otimes_{A^e} \big( S^{-2}y_5 \cdot (S^{-3}y_2 \cdot m) \big) \\
		& \qquad \quad y_6 \cdot \big( \gamma(S^{-1}y_1)y_0 \big) \\
		& = \sum x\gamma^{-1}(S^{-3}y_3) \gamma^{-1}(S^{-2}y_4) \otimes \Big( y_6 \cdot \big( \gamma(S^{-1}y_1)y_0 \big) \Big) \gamma(y_7) \otimes_{A^e}  \\
		& \qquad \big( S^{-2}y_5 \cdot (S^{-3}y_2 \cdot m) \big) \\
		& = \sum x\gamma^{-1}(S^{-3}y_3) \gamma^{-1}(S^{-2}y_4) \otimes \gamma(y_6) \gamma(S^{-1}y_1)y_0 \gamma^{-1}(y_7) \gamma(y_8) \otimes_{A^e} \\
		& \qquad \quad \big( S^{-2}y_5 \cdot (S^{-3}y_2 \cdot m) \big) \\
		& \xlongequal{\eqref{h-cdot-k-cdot-u}} \sum x\gamma^{-1}(S^{-3}y_5) \gamma^{-1}(S^{-2}y_6) \otimes \gamma(y_{10}) \gamma(S^{-1}y_1)y_0 \otimes_{A^e} \\
		& \qquad \quad \sigma(S^{-2}y_7, S^{-3}y_4) \big( (S^{-2}y_8) (S^{-3}y_3) \cdot m \big) \sigma^{-1}(y_9, S^{-1}y_2) \\
		& = \sum x\gamma^{-1}(S^{-3}y_5) \gamma^{-1}(S^{-2}y_6) \sigma(S^{-2}y_7, S^{-3}y_4) \otimes \sigma^{-1}(y_9, S^{-1}y_2) \gamma(y_{10}) \gamma(S^{-1}y_1)y_0 \\
		& \qquad \otimes_{A^e} \big( (S^{-2}y_8) (S^{-3}y_3) \cdot m \big)  \\
		& \xlongequal{\eqref{gamma^{-1}-gamma^{-1}}} \sum x\gamma^{-1}\big( (S^{-2}y_4) (S^{-3}y_3) \big) \otimes \gamma\big( y_{6} (S^{-1}y_1) \big) y_0 \otimes_{A^e} \big( (S^{-2}y_5) (S^{-3}y_2) \cdot m \big) \\
		& = \sum  x \otimes y \otimes_{A^e} (1_H \cdot m) \\
		& \xlongequal{\eqref{1_H-cdot-u}} \sum  x \otimes y \otimes_{A^e} m.
	\end{align*}
	This completes the proof.
\end{proof}

The Nakayama automorphisms of the crossed product $A \#_{\sigma} H$ are determined by those of $A$ and $H$, along with the homological determinant of the $H$-action on $A$.

\begin{thm}\label{Nak-of-cros-prod}
	Let $H$ be a skew Calabi--Yau algebra of dimension $n$, $B = A \#_{\sigma} H$ be a crossed product of $A$ with $H$.
	If $A$ is skew Calabi--Yau of dimension $d$ with a Nakayama automorphism $\mu_A$, then $B$ is also a skew Calabi--Yau algebra of dimension $n+d$ with a Nakayama automorphism $\mu_B$ which is defined by
	$$\mu_B(a\#h) = \sum \mu_A(a) \Hdet(S^{-2}h_1) \# S^{-2}h_{2} \chi(Sh_{3}),$$
	where $\Hdet$ is the homological determinant of the $H$-action on $A$ with respect to a $\mu_A$-twisted volume as defined in Definition \ref{Hdet-defn} and $\chi: H \to \kk$ is given by $\Ext^n_H(\kk, H) \cong {_{\chi}\!\kk}$. 
\end{thm}
\begin{proof}
	Observe that $B$ is a faithfully flat left $A$-module. According to Theorem \ref{sCY-for-ff-H-ext}, $B$ is a skew Calabi--Yau algebra of dimension $n+d$.
	Let $\gamma: H \to B$ be the map defined by $h \mapsto 1 \# h$.
	We have the following $B^e$-module isomorphisms
	\begin{align*}
		& \quad \Ext^{d+n}_{B^e}(B, B^e) & \\
		& \cong \Ext^n_{H}(\kk, \Ext^d_{A^e}(A, B^e)) & \text{ by Corollary \ref{H-Galois-spectral-sequence-for-B^e}} \\
		& \cong \Ext^d_{A^e}(A, B^e) \otimes_H \Ext^n_{H}(\kk, H) \\
		& \cong (B^e \otimes_{A^e} \Ext^d_{A^e}(A, A^e)) \otimes_H \Ext^n_{H}(\kk, H) & \text{ by Proposition \ref{B^e-otimes_{A^e}-Ext^i_{A^e}(M, A^e)-prop} } \\
		& \cong B \otimes_A \Ext^d_{A^e}(A, A^e) \otimes H \otimes_H \Ext^n_{H}(\kk, H) & \text{ by Lemma \ref{omega_B=Botimes_AU} } \\
		& \cong B \otimes_A \Ext^d_{A^e}(A, A^e) \otimes H \otimes_H {_{\chi}\!\kk} \\
		& \xlongequal{\eqref{H-mod-on-B-otimes-M-otimes-H}} B \otimes_A \Ext^d_{A^e}(A, A^e){_{\chi S^{-1}}},
	\end{align*}
	where the $B^e$-module structure on $B \otimes_A \Ext^d_{A^e}(A, A^e)_{\chi S^{-1}}$ is given by
	\begin{align*}
		& \quad (x \otimes y) ( b \otimes_{A} m) \\
		& = \sum xb \gamma^{-1}(S^{-3}y_3) \otimes_A (S^{-3}y_2 \cdot m) \gamma(S^{-1}y_1)y_0 \chi\big( S^{-1}(hy_4) \big) \\
		& = \sum xb \gamma^{-1}(S^{-3}y_3) \otimes_A (S^{-3}y_2 \cdot m) \gamma(S^{-1}y_1)y_0 \chi\big( S(hy_4) \big)
	\end{align*}
	for any $x, y, b \in B$ and $m \in \Ext^d_{A^e}(A, A^e)$.
	Let $e$ be a $\mu_A$-twisted volume on $A$.
	Then $1 \otimes e$ is a generator of $B \otimes_A \Ext^d_{A^e}(A, A^e)_{\chi S^{-1}}$ as a left $B$-module.
	We obtain
	\begin{align*}
		(1 \otimes e) b & = \sum \gamma^{-1}(S^{-3}b_3) \otimes_A (S^{-3}b_2 \cdot e)\gamma(S^{-1}b_1)b_0 \chi(Sb_4) \\
		& = \sum \chi(Sb_4) \gamma^{-1}(S^{-3}b_3) \otimes_A \varphi(S^{-3}b_2) \mu_A\big( \gamma(S^{-1}b_1)b_0 \big) e \\
		& = \sum \chi(Sb_4) \gamma^{-1}(S^{-3}b_3) \varphi(S^{-3}b_2) \mu_A\big( \gamma(S^{-1}b_1)b_0 \big) \otimes_A e.
	\end{align*}
	Hence, $\Ext_{B^e}^{d+n}(B, B^e)$ is isomorphic to $B_{\mu_B}$ as $B$-$B$-bimodules, where $\mu_B$ is an algebra morphism of $B$ defined by
	\[ \mu_B(b) = \sum \chi(Sb_4) \gamma^{-1}(S^{-3}b_3) \varphi(S^{-3}b_2) \mu_A\big( \gamma(S^{-1}b_1)b_0 \big). \]
	
	Then the Nakayama automorphism $\mu_B$ be represented by the homological determinant $\Hdet$ and Nakayama automorphism $\mu_A$ and $\mu_H$ as follow:
	\begin{align*}
		& \quad \mu_B(a \# h) \\
		& = \sum \chi(Sh_5) \gamma^{-1}(S^{-3}h_4) \varphi(S^{-3}h_3) \mu_A\big( \gamma(S^{-1}h_2)(a\#h_1) \big) \\
		& \xlongequal[]{\eqref{inverse-of-gamma}} \sum \chi(Sh_7) \big( \sigma^{-1}(S^{-2}h_5, S^{-3}h_4)\#S^{-2}h_6 \big) \big( \varphi(S^{-3}h_3)\#1 \big) \mu_A\big( (1\#S^{-1}h_2) (a\#h_1) \big) \\
		& = \sum \sigma^{-1}(S^{-2}h_6, S^{-3}h_5) S^{-2}h_7 \cdot \Big( \varphi(S^{-3}h_4) \mu_A\big( (S^{-1}h_3 \cdot a) \sigma(S^{-1}h_2, h_1) \big) \Big) \\
		& \qquad \# S^{-2}h_{8} \chi(Sh_{9}) \\
		& = \sum \sigma^{-1}(S^{-2}h_6, S^{-3}h_5) S^{-2}h_7 \cdot \Big( \varphi(S^{-3}h_4) \mu_A(S^{-1}h_3 \cdot a) \mu_A\big( \sigma(S^{-1}h_2, h_1) \big) \Big) \\
		& \qquad \# S^{-2}h_{8} \chi(Sh_{9}) \\
		& \xlongequal[]{\eqref{varphi-eq1}} \sum \sigma^{-1}(S^{-2}h_6, S^{-3}h_5) S^{-2}h_7 \cdot \Big( \big( S^{-3}h_4 \cdot \mu_A(a) \big) \varphi(S^{-3}h_3) \mu_A\big( \sigma(S^{-1}h_2, h_1) \big) \Big) \\
		& \qquad \quad \# S^{-2}h_{8} \chi(Sh_{9}) \\
		& = \sum \sigma^{-1}(S^{-2}h_6, S^{-3}h_5) S^{-2}h_7 \cdot \big( S^{-3}h_4 \cdot \mu_A(a) \big)  \\
		& \qquad S^{-2}h_8 \cdot \Big( \varphi(S^{-3}h_3) \mu_A\big( \sigma(S^{-1}h_2, h_1) \big) \Big) \# S^{-2}h_{9} \chi(Sh_{10}) \\
		& \xlongequal[]{\eqref{h-k-a}} \sum \sigma^{-1}(S^{-2}h_8, S^{-3}h_7) \sigma(S^{-2}h_9, S^{-3}h_6) \big( (S^{-2}h_{10}) (S^{-3}h_5) \big) \cdot \mu_A(a) \\
		& \qquad \quad \sigma^{-1}(S^{-2}h_{11}, S^{-3}h_4) S^{-2}h_{12}\cdot \Big( \varphi(S^{-3}h_3) \mu_A\big( \sigma(S^{-1}h_2, h_1) \big) \Big) \# S^{-2}h_{13} \chi(Sh_{14}) \\
		& = \sum \mu_A(a) \sigma^{-1}(S^{-2}h_{5}, S^{-3}h_4) S^{-2}h_{6} \cdot \Big( \varphi(S^{-3}h_3) \mu_A\big( \sigma(S^{-1}h_2, h_1) \big) \Big) \# S^{-2}h_{7} \chi(Sh_{8}) \\
		& \xlongequal[]{\eqref{Hdet-definition}} \sum \mu_A(a) \Hdet(S^{-2}h_1) \# S^{-2}h_{2} \chi(Sh_{3}). 
	\end{align*}

	Since $\Ext_{B^e}^{d+n}(B, B^e)$ is an invertible $B$-$B$-bimodule, it follows that $\mu_B$ is an isomorphism.
	Thus, we arrive at the desired conclusion.
\end{proof}


In the following, let's consider the Galois extension over a connected Hopf algebra, where the coradical has dimension one.

\begin{lem}\cite[Proposition 1.3]{Bel2000}\label{connected-Hopf-Galois}
	Let $B$ be a right comodule algebra over a connected Hopf algebra $H$.
	Then the following statements are equivalent.
	
	(1) $B^{co H} \subseteq B$ is a cleft extension.
	
	(2) $B^{co H} \subseteq B$ is a faithfully flat $H$-Galois extension.
	
	(3) There is a total integral $\gamma: H \to B$, that is, $\gamma$ is $H$-colinear and $\gamma(1) = 1$.
\end{lem}

\begin{cor}\label{CY-connected-Hopf-cleft-extension}
	Assume that $\kk$ is of characteristic zero.
	Let $H$ be a connected Hopf algebra of finite GK dimension $d$, and $B$ be a cleft $H$-Galois extension of $\kk$.
	Then $B$ is a skew Calabi--Yau algebra of dimension $d$.
	Moreover, $B$ is Calabi--Yau if and only if $H$ is Calabi--Yau.
\end{cor}
\begin{proof}
	By \cite[Proposition 6.6]{Zho2024}, 
	$H$ is skew Calabi--Yau of dimension $d$.
	Let $\mu_H$ denote the Nakayama automorphism of $H$ as defined in Theorem \ref{Hopf-VdB<->sCY-thm}.
	Given that $B^{co H} = \kk$, the homological determinant is trivial, that is, $\Hdet = \vep$.
	Consequently, Theorem \ref{Nak-of-cros-prod} implies that $B$ is a skew Calabi--Yau algebra of dimension $d$ with a Nakayama automorphism $\mu_B$ defined as
	\[ \mu_B(h) = \sum 1 \# S^{-2}h_{1} \chi(Sh_{2}) = 1 \# \mu_H(h) \in \kk \#_{\sigma} H. \]
	
	Consider the coradical filtration $\{ H_i \}_{i \in \N}$ of the connected Hopf algebra $H$, and let $\gr H$ denote the associated graded algebra $\bigoplus\limits_{i \in \N} H_{i}/H_{i-1}$ with $H_{-1} = 0$.
	Since $\gr H$ is a connected graded Hopf algebra, it is an iterated Ore extension of $\kk$ by \cite[Theorem B]{ZSL2020}, hence a connected graded domain. This indicates that every invertible element of $H$ lies in $H_0 = \kk$.
	Therefore, $H$ is Calabi--Yau if and only if $\mu_H$ is the identity map on $H$.
	Define $\gamma: H \to B$ by $h \mapsto 1 \# h$.
	It is clear that $\{ \gamma(H_i) \}_{i \in \N}$ forms an algebra filtration of $B$.
	As $\Delta(H_n) \subseteq \sum\limits_{i=0}^n H_i \otimes H_{n-i}$, it follows that $\Delta(h) - 1 \otimes h \in H_n \otimes H_{n-1}$ for any $h \in H_n$.
	Consequently,
	\[ \gamma(h) \gamma(k) - \gamma(hk) = \sum \sigma(h_1, k_1) \# h_2k_2 - 1 \# hk \in \gamma(H_{i+j-1}) \]
	for any $h \in H_i \setminus H_{i-1}$ and $k \in H_j \setminus H_{j-1}$.
	Consequently, the associated graded algebra $\gr B$ is also a connected graded domain.
	So every invertible element of $B$ also lies in $\kk$.
	This implies that $B$ is Calabi--Yau if and only if $\mu_B$ is the identity map on $B$, establishing the second assertion.
\end{proof}


	Let $(\mathfrak{g}, [-, -])$ be a finite dimensional Lie algebra, and $f \in Z^2(\mathfrak{g}, \kk)$ be a 2-cocycle, satisfying 
	$$f(x, x) = 0 \;\; \text{ and } \;\; f(x, [y, z]) + f(y, [z, x]) + f(z, [x, y]) = 0$$
	for all $x, y, z \in \mathfrak{g}$.
	The Sridharan enveloping algebra of $\mathfrak{g}$ with respect to $f$ is defined to be the associative algebra
	$$\mathcal{U}_f(\mathfrak{g}) := T\mathfrak{g} / \langle x \otimes y - y \otimes x - [x, y] - f(x, y) \mid x, y \in \mathfrak{g} \rangle.$$
	There exists a canonical $\kk$-linear map $\overline{(-)}: \mathfrak{g} \to \mathcal{U}(\mathfrak{g}), \; x \mapsto \overline{x}$ such that $\mathcal{U}_f(\mathfrak{g})$ is a $\mathcal{U}(\mathfrak{g})$-comodule algebra via
	$$\rho: \mathcal{U}_f(\mathfrak{g}) \To \mathcal{U}_f(\mathfrak{g}) \otimes \mathcal{U}(\mathfrak{g}), \qquad \overline{x} \mapsto \overline{x} \otimes 1 + 1 \otimes x, \text{ for all } x \in \mathfrak{g}.$$
	Given a basis $\{ x_1, x_2, \dots, x_d \}$ of $\mathfrak{g}$, a basis for $\mathcal{U}_f(\mathfrak{g})$ is given by $\{ 1 \}$ and all monomials $\overline{x}_{i_1} \overline{x}_{i_2} \cdots \overline{x}_{i_n}$ with $ i_1 \leq i_2 \leq \dots \leq i_n$.
	Since the map
	\[ \gamma: \mathcal{U}(\mathfrak{g}) \To \mathcal{U}_f(\mathfrak{g}), \qquad x_{i_1} x_{i_2} \cdots x_{i_n} \mapsto \overline{x}_{i_1} \overline{x}_{i_2} \cdots \overline{x}_{i_n} \]
	is an $H$-comodule map,
	Lemma \ref{connected-Hopf-Galois} implies that $\gamma$ is convolution invertible, hence, $\kk \subseteq \mathcal{U}_f(\mathfrak{g})$ is a cleft extension \cite{JS2006}.
	Observe that $\mathcal{U}(\mathfrak{g})$ is a skew Calabi--Yau algebra of dimension $\dim \mathfrak{g}$, as stated in \cite[Theorem A]{Yek2000}.
	Moreover, $\mathcal{U}(\mathfrak{g})$ is Calabi--Yau if and only if $\mathfrak{g}$ is unimodular, which means $\tr(\mathrm{ad}_{\mathfrak{g}}(x)) = 0$ for all $x \in \mathfrak{g}$.
	Consequently, Corollary \ref{CY-connected-Hopf-cleft-extension} yields the following conclusion.
	
\begin{cor}\cite[Theorem 5.3]{HVOZ2010}\cite[Corollary 1.1]{WZ2013}
	The Sridharan enveloping algebra $\mathcal{U}_f(\mathfrak{g})$ is Calabi--Yau if and only if the universal enveloping algebra $\mathcal{U}(\mathfrak{g})$ is Calabi--Yau if and only if $\mathfrak{g}$ is unimodular.
\end{cor}

\section*{Acknowledgments}
The author thanks the referee for the valuable comments and suggestions.
The author is very grateful to Professor Liyu Liu for his helpful suggestion.
The research work was supported by the the NSFC (project 12301052) and Fundamental Research Funds for the Central Universities (project 2022110884). 


\bibliographystyle{siam}
\bibliography{References}

\begin{thebibliography}{10}

\bibitem{Agu2000}
{\sc M.~Aguiar}, {\em A note on strongly separable algebras}, vol.~65, 2000,
  pp.~51--60.
\newblock Colloquium on Homology and Representation Theory (Spanish)
  (Vaquer\'{\i}as, 1998).

\bibitem{AS1987}
{\sc M.~Artin and W.~F. Schelter}, {\em Graded algebras of global dimension
  {$3$}}, Adv. in Math., 66 (1987), pp.~171--216.

\bibitem{Bel2000}
{\sc A.~D. Bell}, {\em Comodule algebras and {G}alois extensions relative to
  polynomial algebras, free algebras, and enveloping algebras}, Comm. Algebra,
  28 (2000), pp.~337--362.

\bibitem{BCM1986}
{\sc R.~J. Blattner, M.~Cohen, and S.~Montgomery}, {\em Crossed products and
  inner actions of {H}opf algebras}, Trans. Amer. Math. Soc., 298 (1986),
  pp.~671--711.

\bibitem{BZ2008}
{\sc K.~A. Brown and J.~J. Zhang}, {\em Dualising complexes and twisted
  {H}ochschild (co)homology for {N}oetherian {H}opf algebras}, J. Algebra, 320
  (2008), pp.~1814--1850.

\bibitem{Brz1996}
{\sc T.~Brzezi\'{n}ski}, {\em Translation map in quantum principal bundles}, J.
  Geom. Phys., 20 (1996), pp.~349--370.

\bibitem{BH2004}
{\sc T.~Brzezi\'{n}ski and P.~M. Hajac}, {\em The {C}hern-{G}alois character},
  C. R. Math. Acad. Sci. Paris, 338 (2004), pp.~113--116.

\bibitem{CWZ2014}
{\sc K.~Chan, C.~Walton, and J.~Zhang}, {\em Hopf actions and {N}akayama
  automorphisms}, J. Algebra, 409 (2014), pp.~26--53.

\bibitem{Ste1995}
{\sc D.~\c{S}tefan}, {\em Hochschild cohomology on {H}opf {G}alois extensions},
  J. Pure Appl. Algebra, 103 (1995), pp.~221--233.

\bibitem{DT1986}
{\sc Y.~Doi and M.~Takeuchi}, {\em Cleft comodule algebras for a bialgebra},
  Comm. Algebra, 14 (1986), pp.~801--817.

\bibitem{DT1989}
\leavevmode\vrule height 2pt depth -1.6pt width 23pt, {\em Hopf-{G}alois
  extensions of algebras, the {M}iyashita-{U}lbrich action, and {A}zumaya
  algebras}, J. Algebra, 121 (1989), pp.~488--516.

\bibitem{DNR2001}
{\sc S.~D\u{a}sc\u{a}lescu, C.~N\u{a}st\u{a}sescu, and c.~Raianu}, {\em Hopf
  algebras}, vol.~235 of Monographs and Textbooks in Pure and Applied
  Mathematics, Marcel Dekker, Inc., New York, 2001.
\newblock An introduction.

\bibitem{Gin2007}
{\sc V.~Ginzburg}, {\em Calabi-yau algebras},  (2007).

\bibitem{HVOZ2010}
{\sc J.-W. He, F.~Van~Oystaeyen, and Y.~Zhang}, {\em Cocommutative
  {C}alabi-{Y}au {H}opf algebras and deformations}, J. Algebra, 324 (2010),
  pp.~1921--1939.

\bibitem{JS2006}
{\sc P.~Jara and D.~\c{S}tefan}, {\em Hopf-cyclic homology and relative cyclic
  homology of {H}opf-{G}alois extensions}, Proc. London Math. Soc. (3), 93
  (2006), pp.~138--174.

\bibitem{Kay2007}
{\sc A.~Kaygun}, {\em Hopf-{H}ochschild (co)homology of module algebras},
  Homology Homotopy Appl., 9 (2007), pp.~451--472.

\bibitem{KKZ2009}
{\sc E.~Kirkman, J.~Kuzmanovich, and J.~J. Zhang}, {\em Gorenstein subrings of
  invariants under {H}opf algebra actions}, J. Algebra, 322 (2009),
  pp.~3640--3669.

\bibitem{KT1981}
{\sc H.~F. Kreimer and M.~Takeuchi}, {\em Hopf algebras and {G}alois extensions
  of an algebra}, Indiana Univ. Math. J., 30 (1981), pp.~675--692.

\bibitem{LMeu2019}
{\sc P.~Le~Meur}, {\em Smash products of {C}alabi-{Y}au algebras by {H}opf
  algebras}, J. Noncommut. Geom., 13 (2019), pp.~887--961.

\bibitem{LWZ2012}
{\sc L.-Y. Liu, Q.-S. Wu, and C.~Zhu}, {\em Hopf action on {C}alabi-{Y}au
  algebras}, in New trends in noncommutative algebra, vol.~562 of Contemp.
  Math., Amer. Math. Soc., Providence, RI, 2012, pp.~189--209.

\bibitem{LWZ2007}
{\sc D.-M. Lu, Q.-S. Wu, and J.~J. Zhang}, {\em Homological integral of {H}opf
  algebras}, Trans. Amer. Math. Soc., 359 (2007), pp.~4945--4975.

\bibitem{LOWY2018}
{\sc J.~L\"{u}, S.-Q. Oh, X.~Wang, and X.~Yu}, {\em A note on the bijectivity
  of the antipode of a {H}opf algebra and its applications}, Proc. Amer. Math.
  Soc., 146 (2018), pp.~4619--4631.

\bibitem{LMZ2017}
{\sc J.-F. L\"{u}, X.-F. Mao, and J.~J. Zhang}, {\em Nakayama automorphism and
  applications}, Trans. Amer. Math. Soc., 369 (2017), pp.~2425--2460.

\bibitem{Mon1993}
{\sc S.~Montgomery}, {\em Hopf algebras and their actions on rings}, vol.~82 of
  CBMS Regional Conference Series in Mathematics, Published for the Conference
  Board of the Mathematical Sciences, Washington, DC; by the American
  Mathematical Society, Providence, RI, 1993.

\bibitem{NT1960}
{\sc T.~Nakayama and T.~Tsuzuku}, {\em On {F}robenius extensions. {I}}, Nagoya
  Math. J., 17 (1960), pp.~89--110.

\bibitem{NT1961}
\leavevmode\vrule height 2pt depth -1.6pt width 23pt, {\em On {F}robenius
  extensions. {II}}, Nagoya Math. J., 19 (1961), pp.~127--148.

\bibitem{RRZ2014}
{\sc M.~Reyes, D.~Rogalski, and J.~J. Zhang}, {\em Skew {C}alabi-{Y}au algebras
  and homological identities}, Adv. Math., 264 (2014), pp.~308--354.

\bibitem{Row1988}
{\sc L.~H. Rowen}, {\em Ring theory. {V}ol. {I}}, vol.~127 of Pure and Applied
  Mathematics, Academic Press, Inc., Boston, MA, 1988.

\bibitem{Sch1998}
{\sc P.~Schauenburg}, {\em Bialgebras over noncommutative rings and a structure
  theorem for {H}opf bimodules}, Appl. Categ. Structures, 6 (1998),
  pp.~193--222.

\bibitem{SS2005}
{\sc P.~Schauenburg and H.-J. Schneider}, {\em On generalized {H}opf {G}alois
  extensions}, J. Pure Appl. Algebra, 202 (2005), pp.~168--194.

\bibitem{Sch1990}
{\sc H.-J. Schneider}, {\em Principal homogeneous spaces for arbitrary {H}opf
  algebras}, Israel J. Math., 72 (1990), pp.~167--195.

\bibitem{VdB2004}
{\sc M.~van~den Bergh}, {\em Non-commutative crepant resolutions}, in The
  legacy of {N}iels {H}enrik {A}bel, Springer, Berlin, 2004, pp.~749--770.

\bibitem{Wei1994}
{\sc C.~A. {Weibel}}, {\em {An introduction to homological algebra}}, vol.~38,
  Cambridge: Cambridge University Press, 1994.

\bibitem{WZ2011}
{\sc Q.-S. Wu and C.~Zhu}, {\em Skew group algebras of {C}alabi-{Y}au
  algebras}, J. Algebra, 340 (2011), pp.~53--76.

\bibitem{WZ2013}
{\sc Q.~S. Wu and C.~Zhu}, {\em Poincar\'{e}-{B}irkhoff-{W}itt deformation of
  {K}oszul {C}alabi-{Y}au algebras}, Algebr. Represent. Theory, 16 (2013),
  pp.~405--420.

\bibitem{Yek2000}
{\sc A.~Yekutieli}, {\em The rigid dualizing complex of a universal enveloping
  algebra}, J. Pure Appl. Algebra, 150 (2000), pp.~85--93.

\bibitem{Yek2020}
\leavevmode\vrule height 2pt depth -1.6pt width 23pt, {\em Derived categories},
  vol.~183 of Cambridge Studies in Advanced Mathematics, Cambridge University
  Press, Cambridge, 2020.

\bibitem{YZ2006}
{\sc A.~Yekutieli and J.~J. Zhang}, {\em Homological transcendence degree},
  Proc. London Math. Soc. (3), 93 (2006), pp.~105--137.

\bibitem{Yu2016}
{\sc X.~Yu}, {\em Hopf-{G}alois objects of {C}alabi-{Y}au {H}opf algebras}, J.
  Algebra Appl., 15 (2016), pp.~1650194, 19.

\bibitem{YZ2016}
{\sc X.~Yu, F.~Van~Oystaeyen, and Y.~Zhang}, {\em Cleft extensions of {K}oszul
  twisted {C}alabi-{Y}au algebras}, Israel J. Math., 214 (2016), pp.~785--829.

\bibitem{YZ2013}
{\sc X.~Yu and Y.~Zhang}, {\em Calabi-{Y}au pointed {H}opf algebras of finite
  {C}artan type}, J. Noncommut. Geom., 7 (2013), pp.~1105--1144.

\bibitem{Zho2024}
{\sc G.-S. Zhou}, {\em Coideal subalgebras of pointed and connected {H}opf
  algebras}, Trans. Amer. Math. Soc., 377 (2024), pp.~2663--2709.

\bibitem{ZSL2020}
{\sc G.-S. Zhou, Y.~Shen, and D.-M. Lu}, {\em The structure of connected
  (graded) {H}opf algebras}, Adv. Math., 372 (2020), pp.~107292, 31.

\end{thebibliography}

\end{document}